\documentclass{article}[11pt]
\usepackage{amsthm,amssymb,amsmath}

\newcommand{\cross}{\times}
\newcommand{\g}{{\frak g}}
\newcommand{\s}{{\frak s}}
\newcommand{\gb}{{\frak b}}
\newcommand{\ga}{{\frak a}}
\newcommand{\cA}{{\cal A}}

\newcommand{\cG}{{\cal G}}
\newcommand{\cF}{{\cal F}}
\newcommand{\bA}{{\bf A}}

\newcommand{\End}{{\cal E}{\it nd}}
\newcommand{\Aut}{{\cal A}{\it ut}}
\newcommand{\St}{{\cal S}{\it t}}
\newcommand{\Tr}{{\cal T}{\it r}}

\newcommand{\Tw}{{\bf Tw}}
\newcommand{\Gal}{{\bf Gal}}
\newcommand{\CYB}{{\cal CYB}}
\newcommand{\Tens}{{\bf Tens}}
\newtheorem{theo}[subsubsection]{Theorem}
\newtheorem{lem}[subsubsection]{Lemma}
\newtheorem{prop}[subsubsection]{Proposition}
\newtheorem{cor}[subsubsection]{Corollary}
\newtheorem{exam}[subsubsection]{Example}
\newtheorem{rem}[subsubsection]{Remark}

\usepackage[all,gar]{xy}
\gardefaultform{@2} \gardefaultlength{1em}

\begin{document}
\author{A. Davydov}
\title{Twisted automorphisms of Hopf algebras}
\maketitle
\begin{center}
Department of Mathematics, Division of Information and Communication Sciences, Macquarie University, Sydney, NSW 2109,
Australia
\end{center}
\begin{center}
davydov@math.mq.edu.au
\end{center}
\begin{abstract}
Twisted homomorphisms of bialgebras are bialgebra homomorphisms from the first into Drinfeld twistings of the second.
They possess a composition operation extending composition of bialgebra homomorphisms. Gauge transformations of twists,
compatible with adjacent homomorphisms, give rise to gauge transformation of twisted homomorphisms, which behave nicely
with respect to compositions. Here we study (gauge classes of) twisted automorphisms of cocommutative Hopf algebras.
After revising well-known relations between twists, twisted forms of bialgebras and $R$-matrices (for commutative
bialgebras) we describe twisted automorphisms of universal enveloping algebras.
\end{abstract}
\section{Introduction}
The aim of the paper is to study ``hidden" symmetries of bialgebras which manifest themselves in representation theory.
It is very well-known that categories of representations (modules) of bialgebras are examples of so-called {\em
monoidal categories}. It is less acknowledged that relations between bialgebras ({\em homomorphisms of bialgebras}) do
not capture all relations ({\em monoidal functors}) between their representation categories. An algebraic notion which
does the job ({\em bi-Galois (co)algebra}) is known only to specialists. They fully represent monoidal relations between
representation categories but are sometimes not very easy to work with. For example, composition of monoidal functors
corresponds to tensor product of Galois (co)algebras and in some situations is quite tricky to calculate explicitly. At
the same time monoidal functors of interest could have some additional properties which put restrictions on the
corresponding algebraic objects and allow one to have an alternative and perhaps simpler description.

Note that any representation category is equipped for free with a monoidal functor to the category of vector spaces,
the functor forgetting the action of the bialgebra ({\em forgetful functor}). Here we deal with monoidal functors
between representation categories which preserve (not necessarily monoidally) the forgetful functors. The corresponding
algebraic notion is of {\em twisted homomorphism}. After defining them in the first section we examine algebraic
counterparts of the composition (composition of twisted homomorphisms) and natural transformations of monoidal functors
({\em gauge transformations}).

The main object to study for us is the category of twisted automorphisms of a cocommutative Hopf algebra. We start by
re-examining the actions of twisted automorphisms on twists (twisted homomorphisms from the ground field) and
$R$-matrices. After that we treat the case of a universal enveloping algebra of a Lie algebra (over formal power
series). It turns out that any twisted automorphism is a bialgebra automorphism together with an invariant twist ({\em
separated} case). The gauge classes of invariant twists form an abelian group isomorphic to the invariant elements of
the exterior square of the Lie algebra. The group of gauge classes of twisted automorphisms is a crossed product of the
group of automorphisms of the Lie algebra and the group of invariant twists.

Throughout the paper $k$ be a ground field, which is supposed to be algebraically closed of characteristic zero.

\section*{Acknowledgment}
The paper was started during the author's visit to the Max-Planck Institut f\" ur Mathematik (Bonn). The author would
like to thank this institution for hospitality and excellent working conditions. The work on the paper was supported by
Australian Research Council grant DP00663514. The author thanks V. Lyakhovsky and A. Stolin for stimulating discussions.
Special thanks are due to R. Street for invaluable support during the work on the paper.

\section{Twisted homomorphisms of bialgebras and monoidal functors between categories of modules}

\subsection{Twisted homomorphisms of bialgebras}
Here we recall the notions of twisted homomorphisms of bialgebras and
their transformations and show how they enrich the category of bialgebras making it a 2-category.

A {\em twisted homomorphism} of bialgebras $(H,\Delta,\varepsilon)\to (H',\Delta',\varepsilon')$ is a pair $(f,F)$
where $f:H\to H'$ is a homomorphism of algebras and $F$ is an invertible element of $H'\otimes H'$ ({\em $f$-twist} or
simply {\em twist}) such that
\begin{equation}\label{conj}
F\Delta'(f(x)) = (f\otimes f)\Delta(x)F,\quad \forall x\in H,
\end{equation}
$$(F\otimes 1)(\Delta'\otimes I)(F) = (1\otimes F)(I\otimes\Delta')(F),\quad \mbox{2-cocycle condition}$$
$$(\varepsilon\otimes I)(F) = (I\otimes\varepsilon)(F) = 1,\quad \mbox{normalisation}.$$

For example, a homomorphism of bialgebras $f:H\to H'$ is a twisted homomorphism with the identity twist $(f,1)$. A
twisted homomorphism is {\em separated} if the first component $f:H\to H'$ is a homomorphism of bialgebras. For a
separated twisted homomorphism the condition (\ref{conj}) amounts to the invariance of the twist with respect to the
sub-bialgebra $f(H)\subset H'$: $$\Delta'(f(x))F = F\Delta'(f(x)),\quad \forall x\in H.$$

Note that the 2-cocycle and normalisation conditions for a twist $F$ imply the coassociativity and counitality
respectively of the {\em twisted} coproduct on $H'$: $$(\Delta')^F(x) = F\Delta'(x)F^{-1}.$$ The condition (\ref{conj})
says that $f:H\to {(H')}_F$ is a homomorphism of bialgebras, where ${H'}_F = (H',(\Delta')^F)$ is the twisted
bialgebra. Note also that the condition on $F\in H'\otimes H'$ equivalent to the coassociativity of the twisted
coproduct $(\Delta')^F$, is $H'$-invariance of $$(\Delta'\otimes I)(F)^{-1}(F\otimes 1)^{-1}(1\otimes
F)(I\otimes\Delta')(F).$$ One could call such elements {\em quasi-twists}. Then a {\em quasi-twisted homomorphism} of
bialgebras $(f,F):H\to H'$ is a bialgebra homomorphisms $f:H\to {H'}_F$ for a quasi-twist $F$. Although, being a
partial case of the general picture of twists for quasi-bialgebras (see \cite{dr}), quasi-twisted homomorphism are
interesting in their own right. For instance, Drinfeld's rectification of Kohno's monodromy theorem establishes the
existence of a quasi-twisted isomorphism between quantised and classical universal enveloping algebras $U_q(\g)[[h]]\to
U(\g)[[h]]$ (see \cite{dr}). For another example see \cite{da}, where it was proved that finite groups have the same
character tables if and only if there is a quasi-twisted isomorphism between their group algebras (over algebraically
closed field of characteristic zero). Being much more versatile, quasi-twisted homomorphisms lack one important
property valid for twisted homomorphisms. There is no composition operation defined for general quasi-twisted
homomorphisms. It is related to the fact that (in contrast to twists) the image of a quasi-twist under a bialgebra
homomorphism is not necessarily a quasi-twist.

For succesive twisted homomorphisms $$\xymatrix{(H,\Delta,\varepsilon)\ar[r]^{(f,F)} &
(H',\Delta',\varepsilon')\ar[r]^{(f',F')} & (H'',\Delta'',\varepsilon'')}$$ we define their {\em composition} as
$$(f',F')\circ (f,F) = (f'f,f'(F)F').$$ Here $f'(F) = (f'\otimes f')(F)$. It is not hard to verify that the result is a
twisted homomorphism $(H,\Delta,\varepsilon)\to (H'',\Delta'',\varepsilon'')$ and that the composition is associative.
See also section \ref{catint}. Note that separated twisted homomorphisms are closed under composition.

By a {\em gauge transformation} $(f,F)\to(f',F')$ of twisted homomorphisms $(f,F),(f',F'):(H,\Delta,\varepsilon)\to
(H',\Delta',\varepsilon')$ we will mean an element $a$ of $H'$ such that
\begin{equation}\label{comgt}
af(x) = f'(x)a,\quad \forall x\in H,
\end{equation}
\begin{equation}\label{mongt}
F\Delta'(a) = (a\otimes a)F'.
\end{equation}
We will depict it graphically as follows: $$\xymatrix{(H,\Delta,\varepsilon) {\ar@/^20pt/[rr]^{(f,F)} }\glabelar{s}
{\ar@/_20pt/[rr]_{(f',F')}  }\glabelar{t} & & (H',\Delta',\varepsilon') \gar[s,t]^{a} } $$ Note that the condition
(\ref{mongt}) together with normalisation conditions for $F$ and $G$ implies $\varepsilon(a) = 1$. Note also that a
twisted homomorphism gauge isomorphic to a separated twisted homomorphism is not necessarily separated.

For successive gauge transformations
$$\xymatrix{(H,\Delta,\varepsilon) {\ar@/^30pt/[rr]^{(f,F)}}\glabelar{s}
{\ar[rr]_{(g,G)}}\glabelar{m} {\ar@/_30pt/[rr]_{(j,J)}}\glabelar{t} & & (H',\Delta',\varepsilon') \gar[s,m]^{a} \gar[m,t]^{b}}$$
the composition $b.a:(f,F)\to(j,J)$ is simply the product $ba$ in $H'$. Again it is quite
straightforward to check that this is a transformation.

Note that if $H'$ is a Hopf algebra, any gauge transformation between twisted homomorphisms into $H'$ is invertible.
This fact can be checked by pure algebraic computations. A sketch of a different proof will be given in section
\ref{catint}.

We can also define compositions of transformations and twisted homomorphisms in the following two situations:
$$\xymatrix{(H,\Delta,\varepsilon) {\ar@/^20pt/[rr]^{(f,F)}}\glabelar{s} {\ar@/_20pt/[rr]_{(f',F')}}\glabelar{t} & &
(H',\Delta',\varepsilon') \ar[rr]^{(g,G)} & & (H'',\Delta'',\varepsilon'') \gar[s,t]^{a} }$$
$$\xymatrix{(H,\Delta,\varepsilon) \ar[rr]^{(f,F)} & & (H',\Delta',\varepsilon') {\ar@/^20pt/[rr]^{(g,G)}}\glabelar{s}
{\ar@/_20pt/[rr]_{(g',G')}}\glabelar{t} & & (H'',\Delta'',\varepsilon'') \gar[s,t]^{b}}$$ we define it to be
$(g,G)\circ a = g(a)$ in the first case and $b\circ(f,F) = b$ in the second. The following properties intertwining
compositions of twisted homomorphisms and gauge transformations are quite straightforward consequences of the
definitions: $$(a.b)\circ(f,F) = (a\circ(f,F)).(b\circ(f,F)),$$ $$(g,G)\circ(a.b) = ((g,G)\circ a).((g,G)\circ b),$$
$$(a\circ(f,F)).((g,G)\circ b) = ((g,G)\circ b).(a\circ(f,F)).$$ For an alternative explanation see section
\ref{catint}.

Note that the structures described above extend the category of bialgebras and twisted homomorphisms to a 2-category
$\Tw$ with gauge transformations as 2-cells (2-morphisms).

Twisted homomorphisms have certain natural involutive symmetry.
\begin{lem}\label{transp}
If $(f,F):H\to H'$ is a twisted homomorphism, then $t(f,F) = (f,t(F)):H^{co}\to (H')^{co}$ is a twisted homomorphism
between bialgebras with opposite comultiplication.
\end{lem}
\begin{proof}
The condition (\ref{conj}) for $(f,t(F))$ is the transpose (the result of applying $t$) of the condition (\ref{conj})
for $(f,F)$. The same with the normalisation condition for $t(F)$. The 2-cocycle equation for $t(F)$ is the 2-cocycle
equation for $F$ where the first and the last tensor factor have been interchanged: $$(t(F)\otimes 1)(t\Delta'\otimes
I)t(F) = (t_2t_1t_2)((1\otimes F)(I\otimes\Delta')(F)) = $$ $$(t_1t_2t_1)((F\otimes 1)(\Delta'\otimes I)(F)) =
(1\otimes t(F))(I\otimes t\Delta')t(F).$$
\end{proof}
Note that if we had a gauge transformation $a:(f,F)\to(g,G)$, then $a:t(f,F)\to t(g,G)$ is also a gauge transformation.
Moreover, all compositions are compatible with $t$. In (2-)categorical language $t:\Tw\to\Tw$ is a 2-functor
(2-isomorphism).

We call a twisted homomorphism $(f,F):H\to H'$ of cocommutative bialgebras {\em symmetric} if $t(f,F) = (f,F)$. Note
that a twisted homomorphism gauge isomorphic to a symmetric twisted homomorphism is symmetric itself.

\subsection{Twisted homomorphisms and Galois (co)algebras}\label{galtw}
It was observed by Drinfeld \cite{dr0} that a twist $F$ on a bialgebra $H$ defines a new coalgebra structure on $H$:
$$\Delta_F(x) = F\Delta(x),\quad x\in H.$$ Coassociativity for $\Delta_F$ is equivalent to the 2-cocycle condition on
$F$. The property $\Delta_F(xy) = \Delta_F(x)\Delta(y)$ guarantees that the right $H$-module structure on $H$ is an
$H$-module coalgebra structure on $(H,\Delta_F)$. Recall that a coalgebra $(C,\delta)$ is  an {\em $H$-module
coalgebra} if $C$ is an $H$-module $a:C\otimes H\to C$ and $$\delta(xy) = \delta(x)\Delta(y),\quad x\in C, y\in H.$$ An
$H$-module coalgebra $C$ is {\em Galois} if the composition $$\xymatrix{ C\otimes H \ar[r]^{\delta\otimes I} & C\otimes
C\otimes H \ar[r]^{I\otimes a} & C\otimes C }$$ is an isomorphism. It is not hard to see that if $H$ is a Hopf algebra
$(H,\Delta_F)$ is a Galois $H$-module coalgebra.  Moreover Galois $H$-module coalgebras of the form $(H,\Delta_F)$ are
characterised by the so-called normal basis property. A Galois $H$-module coalgebra $C$ has a {\em normal basis} if $C$
is isomorphic to $H$ as an $H$-module.

For a twisted homomorphism $(f,F):H\to H'$, the $H'$-Galois coalgebra $(H',\Delta_F)$ comes equipped with the left
$H$-action $H\otimes H'\to H'$ sending $x\otimes y$ to $f(x)y$. It follows from the definition of twisted homomorphism
that this action preserves the coproduct $\Delta_F$: $$F\Delta'(f(x)y) = F\Delta'(f(x))\Delta'(y) = (f\otimes
f)\Delta(x)F\Delta'(y).$$

Composition of twisted homomorphisms corresponds to the following operation on Galois coalgebras. Let $C$ be a Galois
$H'$-module coalgebra and $C'$ be a Galois $H''$-module coalgebra with compatible left $H'$-action. Then the tensor
product of $H'$-modules $C\otimes_{H'}C'$ is a Galois $H''$-coalgebra with respect to the coproduct
$t_{23}(\delta\otimes\delta')$. Moreover a left $H$-action on $C$ compatible with $H'$-module coalgebra structure will
pass on to $C\otimes_{H'}C'$.

For successive twisted homomorphisms $$\xymatrix{(H,\Delta,\varepsilon)\ar[r]^{(f,H)} &
(H',\Delta',\varepsilon')\ar[r]^{(f',F')} & (H'',\Delta'',\varepsilon'')}$$ the tensor product
$(H',\Delta_F)\otimes_{H'}(H'',\Delta_{F'})$ is isomorphic to $H''$ as an $H-H''$-bimodule via $x\otimes y\mapsto
f'(x)y$. Moreover, the tensor product of comultiplications $\Delta_F\otimes\Delta_{F'}$ is carried out into the
comultiplication $\Delta_{f(F)F'}$. Indeed, the tensor product of comultiplications followed by the isomorphism sends
$x\otimes y$ into $(f'\otimes f')(F\Delta'(x))F'\Delta''(y)$ which can be rewritten as $$(f'\otimes
f')(F\Delta'(x))F'\Delta''(y) = (f'\otimes f')(F)(f'\otimes f')(\Delta'(x))F'\Delta''(y) = $$ $$(f'\otimes
f')(F)F'\Delta''(f'(x))\Delta''(y) = (f'\otimes f')(F)F'\Delta''(f'(x)y).$$ Finally $(f'\otimes
f')(F)F'\Delta''(f'(x)y)$ is the coproduct $\Delta_{f(F)F'}$ applied to $f'(x)y$.

A gauge transformation $a\in H'$ of twisted homomorphisms $(f,F),(f',F'):(H,\Delta,\varepsilon)\to
(H',\Delta',\varepsilon')$ defines a homomorphism $x\mapsto ax$ of Galois $H'$-module coalgebras. Preservation of
$H'$-module structure is obvious. Compatible $H$-module structure is preserved by the property (\ref{comgt}) of gauge
transformations, compatibility with coalgebra structures follows from the property (\ref{mongt}): $$F\Delta'(ax) =
F\Delta'(a)\Delta'(x) = (a\otimes a)F'\Delta(x).$$

Define a bicategory $\Gal$ with objects being Hopf algebras, arrows from $H$ to $H'$ being Galois right $H'$-module
coalgebras with compatible left $H$-action and composition given by tensor product, and homomorphisms of
$H-H'$-bimodule coalgebras as 2-cells. The above construction defines a psedofunctor $\Tw\to\Gal$.

In the case of finite dimensional Hopf algebras we can replace Galois coalgebra with the more familiar notion of Galois
algebra. Note that for a finite dimensional Hopf algebra $H$, a Galois $H$-module coalgebra must be finite dimensional
(of the same dimension). Note also that the dual algebra to a Galois $H$-coalgebra will be a Galois $H$-algebra. Recall
that an {\em $H$-algebra} $R$ is an algebra and an $H$-module in such way that the multiplication map $\mu$ is a map of
$H$-modules: $$\mu\Delta(x)(a\otimes b) = x\mu(a\otimes b) = x(ab),\quad x\in H, a,b\in R.$$ The {\em crossed product}
$R*H$ of $H$ with an $H$-algebra $R$ is an algebra, which as a vector space is isomorphic to the tensor product
$R\otimes H$. Denote elements of $R*H$ corresponding to tensors $a\otimes x$ by $a*x$. The product on $R*H$ is given by
the rule: $$(a*x)(b*y) = \sum_{(x)}ax_1*x_2y.$$ Here we use the so-called Swidler's notation for the coproduct
$\Delta(x) = \sum_{(x)}x_1\otimes x_2$. Following \cite{mon} we call an $H$-algebra $R$ {\em Galois} if the
homomorphism of algebras $$A*H\to End(A),\quad a*x\mapsto (b\mapsto ax(a))$$ is an isomorphism.

For a finite dimensional $H$ the dual of the Galois coalgebra $(H,\Delta_H)$ is the dual space $H^*$ with the
multiplication: $$(l*_Fl')(x) = (l\otimes l')\Delta_F(x) = (l\otimes l')(F\Delta(x)).$$ This can be seen as an ordinary
multiplication on $H^*$ twisted by $F$ with respect to the (right) $H$-action on $H^*$: $$l*_Fl' = \sum
l^F_1(l')^F_2.$$ Here $F=\sum F_1\otimes F_2$ and $l^y(x) = l(xy)$.

\section{Categorical interpretations}\label{catint}
Most of the material included in this section is pretty standard (for example, see \cite{es}; for more categorically oriented treatment see \cite{st}). 
\subsection{Monoidal categories and functors}

A {\em monoidal category} is a category $\cG$  with a functor
\begin{displaymath}
\otimes :{\cal G} \times {\cal G} \longrightarrow {\cal G} \qquad (X,Y) \mapsto X \otimes Y
\end{displaymath}
({\em tensor product}), a natural collection of isomorphisms ({\em associativity constraint})
\begin{displaymath}
\varphi_{X,Y,Z} : X \otimes (Y \otimes Z) \rightarrow (X \otimes Y) \otimes Z \qquad \mbox{for any} \quad X,Y,Z \in
{\cal G}
\end{displaymath}
which satisfies the following {\em pentagon axiom}: $$(X\otimes\varphi_{Y,Z,W})\varphi_{X,Y\otimes
Z,W}(\varphi_{X,Y,Z}\otimes W) = \varphi_{X,Y,Z\otimes W}\varphi_{X\otimes Y,Z,W}$$ and an object $1$ ({\em unit}
object) together with natural isomorphisms $$\rho_{X}:X \otimes 1 \rightarrow X \quad \lambda_{X}:1 \otimes X
\rightarrow X$$ such that $\lambda_1 = \rho_1$ $$\lambda_{X\otimes Y} = \lambda_{X}\otimes I:1\otimes X\otimes Y\to
X\otimes Y,\quad \rho_X\otimes I = I\otimes\lambda_Y:X\otimes 1\otimes Y\to X\otimes Y,$$ $$\rho_{X\otimes Y} =
I\otimes\rho_{Y}:X\otimes Y\otimes 1\to X\otimes Y$$ for any $X,Y\in\cal G$. Here $I$ denote the identity morphism. It
is known that the first two conditions follow from the rest. We formulate them for the sake of symmetry.

The celebrated MacLane coherence theorem \cite{mc} says that there is a unique isomorphism between any two bracket
arrangements on the tensor products of objects $X_{1},...,X_{n}$, which is a composition of (tensor products of) the
associativity constraints. This fact allows to omit brackets in tensor products. It is also easy to see that the unit
object is unique up to an isomorphism. More precisely, any monoidal category is monoidally equivalent to a {\em strict}
monoidal category (a monoidal category with identity associativity and unit object constraints) (see, for example,
\cite{js}).

A {\em monoidal functor} between monoidal categories ${\cal G}$ and ${\cal H}$ is a functor $F : {\cal G} \to {\cal H}$
with a natural collection of isomorphisms (the so-called {\em monoidal structure})
\begin{displaymath}
F_{X,Y} : F(X \otimes Y) \rightarrow F(X) \otimes F(Y) \qquad \mbox{for any} \quad X,Y \in {\cal G}
\end{displaymath}
satisfying the {\em coherence} axiom:
\begin{equation}\label{cohmf}
(I\otimes F_{Y,Z})F_{X,Y\otimes Z} = (F_{X,Y}\otimes I)F_{X\otimes Y,Z}
\end{equation}
for any objects $X,Y,Z \in {\cal G}$.

A natural transformation $f : F \rightarrow G$ of monoidal functors $F$ and $G$ is {\em monoidal} if $$G_{X,Y}f_{X
\otimes Y} = (f_{X} \otimes f_{Y})F_{X,Y}$$ for any $X,Y \in {\cal G}$.

Here we will be mostly interested in tensor categories and functors (linear over the ground field $k$). Recall that a
monoidal category is {\em tensor} if it is $k$-linear abelian and the tensor product is $k$-linear and bi-exact. For a
monoidal functor between tensor categories to be {\em tensor} we will ask it to be $k$-linear and left exact. Denote by
$\Tens$ the 2-category of tensor categories, with tensor functors and monoidal natural transformations.

\subsection{Categories of modules over bialgebras}

Recall that the comultiplication in $H$ can be used to define a structure of an $H$-module on the tensor product
$M\otimes_{k}N$ of two $H$-modules: $$h*(m\otimes n) = \Delta(h)(m\otimes n)\qquad h\in H, m\in M, n\in N.$$ The
coassociativity axiom for the coproduct implies that the obvious associativity constraint for vector spaces $$\varphi
:L\otimes (M\otimes N)\rightarrow (L\otimes M)\otimes N \qquad \varphi (l\otimes (m\otimes n)) = (l\otimes m)\otimes
n$$ is $H$-linear. The counit defines an $H$-module structure on the ground field $k$ and the counit axiom guarantees
that this is a unit object. Thus the category $H\-Mod$ of (left) modules over a bialgebra becomes monoidal.

For a homomorphism of algebras $f:H\to H'$ define by $f^*:H'\-Mod\to H\-Mod$ the {\em inverse image} functor, which
turns an $H'$-module $M$ into an $H$-module $f^*(M)$. Here as a vector space, $f^*(M)$ is the same as $M$ but with a
new module structure $x.m = f(x)m$ for $x\in H$ and $m\in M$. On the level of categories of modules, twisted
homomorphisms and gauge transformations take the following meaning.
\begin{prop}\label{catme}
For a twisted homomorphism $(f,F):H\to H'$ the inverse image functor $f^*:H'-Mod\to H-Mod$ becomes tensor, with the
monoidal structure given by multiplication with the twist: $$F_{M,N}:f^*(M\otimes H)\to f^*(M)\otimes f^*(N),\quad
m\otimes n\mapsto F(m\otimes n).$$ Compositions of twisted homomorphisms and corresponding functors are related as
follows: $$((f,F)\circ(g,G))^* = (g,G)^*\circ(f,F)^*.$$ A gauge transformation $a:(f,F)\to (g,G)$ defines a monoidal
natural transformation $a:(f,F)^*\to (g,G)^*$: $$a_M:f^*(M)\to g^*(M),\quad m\mapsto am.$$ Compositions of gauge
transformations correspond to compositions of natural transformations.
\end{prop}
\begin{proof}
It is straightforward to see that the condition (\ref{conj}) guarantees $H$-linearity of the monoidal constraint
$F_{M,N}$ while the 2-cocycle condition for $F$ is equivalent to the coherence axiom (\ref{cohmf}). Similarly, the
condition (\ref{comgt}) for a gauge transformation $a$ says that $a_M$ is a morphism of $H$-modules and the condition
(\ref{mongt}) is equivalent to the monoidality of $a_M$.
\end{proof}

2-categorically the proposition \ref{catme} says that the inverse image functor construction defines a 2-functor
$\Tw\to \Tens$ contravariant on morphisms and covariant on 2-cells. We can extend this to a bifunctor $\Gal\to\Tens$ as
follows.
\begin{prop}
A Galois $H'$-module coalgebra $C$ with compatible $H$-action defines a tensor functor $C\otimes_{H'}\-:H'\-Mod\to
H\-Mod$ with the monoidal structure
 $$\xymatrix{C\otimes_{H'}(M\otimes N) \ar[d]^{\delta\otimes I}  & C\otimes_{H'}M\otimes C\otimes_{H'}N \\
(C\otimes C)\otimes_{H'}(M\otimes N) \ar[r] & (C\otimes C)\otimes_{H'\otimes H'}(M\otimes N) \ar[u]^{t_{23}} }$$
Homomorphisms of $H$-$H'$-bimodule coalgebras define monoidal natural transformations.
\end{prop}
\begin{proof}
It follows from the Galois property that the above monoidal structure is an isomorphism. Coassociativity of the
coproduct implies the coherence axiom.
\end{proof}
Obviously, composition of functors corresponds to the tensor product of coalgebras.

\section{Twisted automorphisms of a Hopf algebra}
In this section we will describe Cat-groups of twisted automorphisms along with some of their natural Cat-subgroups.
One of the reasons (which will be particularly important in the last section) why Cat-groups are more natural objects
to deal with rather than groups of (classes of) twisted automorphisms is the following. Suppose that we want to define
an action of a group $G$ on a bialgebra $H$ by twisted homomorphisms. A homomorphism from the group $G$ into the group
$Out_\Tw(H)$ of classes of twisted homomorphism would just not have enough information. A homomorphism from $G$ to the
group $Aut_\Tw(H)$ of twisted automorphisms would do, but that would not capture all the cases. The right answer is a
(monoidal) functor $G\to\Aut_\Tw(H)$ of Cat-groups (a map of crossed modules of groups).

\subsection{Categorical groups, crossed modules and their maps}\label{catgroups}

Recall that a {\em categorical group} is a monoidal category in which every arrow is invertible (monoidal groupoid) and
for every object $X$ there is an object $X^*$ with an arrow $e_X:X^*\otimes X\to I$ (a dual object). A categorical
group is {\em strict} (or a {\em Cat-group}) if it is strict as a monoidal category and $e_X$ can be chosen to be an
identity. In other words a Cat-group is a categorical group whose objects form a group (with the tensor product). A
slight modification of Mac Lane's coherence theorem \cite{mc} says that any categorical group is monoidally equivalent
to a strict one. Note that Cat-groups are group objects in the category of categories (thus the name).

For any object $A$ of a 2-category $\bA$ the category $\Aut_\bA(A)$ of automorphisms of $A$ (invertible 1-morphisms
$A\to A$), with the composition as the tensor product and bijective natural transformations as morphisms, is a
Cat-group.

It is well known (see \cite{bs} for the history) that Cat-groups are the same as crossed modules of Whitehead
(\cite{wh}). Recall that a {\em crossed module} of groups is a pair of groups $P,C$ with a (left) action of $P$ on $C$
(by group automorphisms): $$P\times C\to C,\quad (p,c)\mapsto ^{p}{c}$$ and a homomorphism of groups $$P
\stackrel{\partial}{\leftarrow} C$$ such that $$\partial(^{p}{c}) = p\partial(c)p^{-1},$$ $$^{\partial(c)}{c'} =
cc'c^{-1}.$$ For a Cat-group $\cG$ the corresponding crossed complex consists of the group of objects $P$, the group of
morphisms $X\to I$ into the identity object with the product: $$(X\stackrel{a}{\to} I).(Y\stackrel{b}{\to} I) =
X\otimes Y \stackrel{a\otimes b}{\longrightarrow} I\otimes I = I,$$ the action $$^{Y}(X\stackrel{a}{\to} I) = Y\otimes
X\otimes Y^* \stackrel{Y\otimes a\otimes Y^*}{\longrightarrow} Y\otimes I\otimes Y^* = I,$$ and the homomorphism
$\partial:C\to P$ sending $X\to I$ into $X$.

Relative simplicity and compactness of the notion of crossed module as a description for Cat-groups has its downside. A
monoidal functor between Cat-groups will correspond to something less obvious (and less well-known) than a {\em
homomorphism} of crossed modules (a pair of group homomorphisms preserving all the structures). That will describe only
{\em strict} monoidal functors. To deal with general monoidal functors we need the following weaker relation. A {\em
map} of crossed modules $(P,C)\to (E,N)$ is a triple $(\tau,\nu,\theta)$ where $\tau$ and $\nu$ are maps making the
diagram commutative $$\xymatrix{P \ar[d]^\tau & C\ar[l]^{\partial} \ar[d]^\nu \\ E & N \ar[l]^{\partial} }$$ and
$\theta:P\times P\to N$ such that $$\tau(pq) = \partial(\theta(p,q))\tau(p)\tau(q),\quad p,q\in P,$$ $$\nu(ab) =
\theta(\partial(a),\partial(b))\nu(a)\nu(b),\quad a,b\in C,$$ $$\theta(p,qr) ^{\tau(p)}{\theta(q,r)} =
\theta(pq,r)\theta(p,q),\quad p,q,r\in P,$$ $$\theta(1,q) = 1 = \theta(p,1),\quad p,q\in P,$$ $$\nu( ^{p}{a}) =
\theta(p,\partial(a)) ^{\tau(p)}{\nu(a)}.$$

Complete invariants of a categorical-group $\cG$ with respect to monoidal equivalences are $$\pi_0(\cG),\ \pi_1(\cG),\
\phi\in H^3(\pi_0(\cG),\pi_1(\cG)),$$ where the first is the group of isomorphism classes of objects, the second is the
abelian group ($\pi_0(\cG)$-module) $Aut_\cG(I)$ of automorphisms of the unit object and the third is a cohomology
class (the {\em associator}). In the crossed module setting $$\pi_0 = coker(\partial),\ \pi_1 = ker(\partial).$$ Note
that the image of $\partial$ is normal so the cokernal has sense. The class $\phi$ is defined as follows: choose a
section $\sigma:coker(\partial)\to P$ and a map $a:coker(\partial)\times coker(\partial)\to C$ such that $$\sigma(fg) =
\partial(a(f,g))\sigma(f)\sigma(g),\quad  f,g\in coker(\partial).$$ Then for any $f,g,h\in coker(\partial)$ the
expression $$a(f,gh) ^{\sigma(f)}{a(g,h)}a(f,g)^{-1}a(fg,h)$$ is always in the kernel  of $\partial$ and is a group
3-cocycle of $coker(\partial)$ with coefficients in $ker(\partial)$. The cohomology class $\phi$ does not depend on the
choices made.

Analogously, complete invariants of a monoidal functor $F:\cG\to\cF$ between categorical groups with respect to
monoidal isomorphisms are $$\pi_0(F):\pi_0(\cG)\to\pi_0(\cF),\ \pi_1(F):\pi_1(\cG)\to\pi_1(\cF),\
\theta(F):\pi_0(\cG)\times\pi_0(\cG)\to \pi_1(\cF),$$ where the first is the homomorphism of groups, the second is the
homomorphism of $\pi_0(\cG)$-modules and the third is really a class in
$C^2(\pi_0(\cG),\pi_1(\cF))/B^2(\pi_0(\cG),\pi_1(\cF))$ such that $$d(\theta(F)) = \pi_1(F)_*(\phi(\cG)) -
\pi_0(F)^*(\phi(\cF))$$ Here $$\pi_0(F)^*:C^*(\pi_0(\cF),\pi_1(\cF))\to C^*(\pi_0(\cG),\pi_1(\cF)),$$
$$\pi_1(F)_*:C^*(\pi_0(\cG),\pi_1(\cG))\to C^*(\pi_0(\cG),\pi_1(\cF))$$ are the maps of cochain complexes induced by
the group homomorphisms $\pi_0(F)$, $\pi_1(F)$.

\subsection{Cat-groups of twisted automorphisms}
As a part of the 2-category $\Tw$ the invertible twisted endomorphisms of a Hopf algebra $H$ (twisted automorphisms)
and gauge transformations between them form a Cat-group $\Aut_{\Tw}(H)$. The corresponding crossed module of groups has
the form $$Aut_{\Tw}(H) \stackrel{\partial}{\leftarrow} H_\varepsilon^\cross.$$ Here $Aut_{\Tw}(H)$ is the group of
twisted automorphisms of $H$ with respect to the composition, $H_\varepsilon^\cross$ is the group of invertible
elements $x$ of $H$ such that $\varepsilon(x)=1$, and $\partial$ sends $x$ into the pair (an {\em inner} twisted
automorphism) $( ^{x}{(\ )},(x\otimes x)\Delta(x)^{-1})$ where the first component is the conjugation automorphism:
$$^{x}{(\ )}:H\to H,\quad ^{x}{(y)} = xyx^{-1}.$$ The action of $Aut_{\Tw}(H)$ on $H^\cross$ is given by the action of
the first component $(f,F)(y) = f(y)$.

There are two important Cat-subgroups in $\Aut_{\Tw}(H)$. The first is the full Cat-subgroup $\Aut^1_{\Tw}(H)$ of
twisted automorphisms with the identity as the first component. Its crossed module is $$Aut^1_{\Tw}(H)
\stackrel{\partial}{\leftarrow} (Z(H)_\varepsilon)^\cross.$$ Here $Aut^1_{\Tw}(H)$ is the group of {\em invariant
twists} on $H$ (invertible elements of $H\otimes H$ commuting with the image $\Delta(H)$ and satisfying the 2-cocycle
condition), $(Z(H)_\varepsilon)^\cross$ is the group of invertible elements of the centre of counit 1:
$\varepsilon(x)=1$ . Again $\partial$ assigns to $x$ the invariant twist $(x\otimes x)\Delta(x)^{-1}$. The action of
$Aut^1_{\Tw}(H)$ on $(Z(H)_\varepsilon)^\cross$ is trivial.

The second is the full Cat-subgroup $\Aut_{bialg}(H)$ of bialgebra automorphisms of $H$. Here the crossed module is
$$Aut_{bialg}(H) \stackrel{\partial}{\leftarrow} G(H),$$ where $Aut_{bialg}(H)$ is the group of automorphisms of $H$ as
a bialgebra, $G(H)=\{x\in H,\quad \Delta(x) = x\otimes x\}$ is the group of group-like elements of $H$ and $\partial$
sends $x$ into the conjugation automorphism. The action of $Aut_{bialg}(H)$ on $G(H)$ is obvious.

Note that the Cat-subgroup $\Aut^1_{\Tw}(H)$ is what might be called {\em normal}: the components of its crossed module
are normal subgroups in the components of the crossed complex for $Aut_{\Tw}(H)$ and the action of $Aut_{\Tw}(H)$ on
$H^\cross$ preserve the subgroup $Z(H)^\cross$.

The Cat-subgroup $\Aut_{bialg}(H)$ is not in general normal. In the next part we will characterise it as the stabiliser
of a certain action. Recall that an {\em action} of a Cat-group (monoidal category) $\cG$ on a category $\cA$ is a
monoidal functor $\cG\to \End(\cA)$ into the category of endofunctors on $\cA$. For an object $A\in\cA$ the {\em
stabiliser} $\St_\cG(A)$ is the category of pairs $(G,g)$, where $G$ is an object and $g:G(A)\to A$ is an isomorphism
in $\cA$. A morphism of pairs $(G,g)\to(F,f)$ is a morphism $x:G\to F$ in $\cG$ such that the diagram $$\xymatrix{ G(A)
\ar[rr]^{x_A} \ar[rd]_g & & F(A) \ar[ld]^f \\ & A & }$$ commute.

\subsection{Action on twists}\label{actw}
Here we examine the (categorical) action of the Cat-group $\Aut_{\Tw}(H)$ on the category of twisted homomorphisms
$\Tw(k,H)$ given by the composition.

Let us start with the category $\Tw(k,H)$. Any twisted homomorphism $k\to H$ must have the form $(\iota,F)$ where
$\iota:k\to H$ is the unit inclusion and $F\in H^{\otimes 2}$ is an invertible element satisfying the 2-cocycle
condition (a {\em twist}). A gauge transformation $(\iota,F)\to (\iota,F')$ is an invertible element $a\in H$ such that
$F' (a\otimes a)= \Delta(a)F$. The category $\Tw(k,H)$ has a marked object $(\iota,1)$.

On the level of objects the action $Aut_{\Tw}(H)\times \Tw(k,H)\to \Tw(k,H)$ has the form $$(f,F)\circ(\iota,F')\mapsto
(\iota,{F'}^{f}F).$$ An object of the stabiliser $\St_{Aut_{\Tw}(H)}(\iota,1)$ is a triple $(f,F,a)$, where $(f,F)$ is
an object of $Aut_{\Tw}(H)$ and $a\in H$ is an invertible element such that $F = (a\otimes a)\Delta(a)^{-1}$ (a
transformation $(f,F)\circ(\iota,1)\to(\iota,1)$). The element $a$ can be interpreted as a gauge transformation
$(f,F)\to ( ^{a}{(\ )}\circ f,1)$. In particular, the composition $^{a}{(\ )}\circ f$ of $f$ with conjugation with $a$
is an automorphism of bialgebras. The naturality of this construction implies the following result.
\begin{prop}
The inclusion $\Aut_{bialg}(H)\to \St_{Aut_{\Tw}(H)}(\iota,1)$ is an equivalence of categories.
\end{prop}

Now we describe the stabilizer $\St_{Aut^1_{\Tw}(H)}(\iota,1)$ in the Cat-subgroup of invariant twists. Its objects are
pairs $(F,a)$ where $F$ is an invariant twist on $H$ and $a\in H$ is an invertible element such that $F = (a\otimes
a)\Delta(a)^{-1}$ (a transformation $(I,F)\circ(\iota,1)\to(\iota,1)$). Similarly, the element $a$ can be seen as a
gauge transformation $(I,F)\to ( ^{a}{(\ )},1)$. In particular, the conjugation $^{a}{(\ )}$ is an automorphism of
bialgebras. Thus we have the following.
\begin{prop}
The stabiliser $\St_{Aut^1_{\Tw}(H)}(\iota,1)$ is equivalent to the Cat-group $\Aut^{inn}_{bialg}(H)$ with the crossed
module:
\begin{equation}\label{crbial}
Aut^{inn}_{bialg}(H) \stackrel{\partial}{\leftarrow} G(H),
\end{equation}
where $Aut^{inn}_{bialg}(H)$ is the group of bialgebra automorphisms of $H$ which are inner as algebra automorphisms,
i.e. the kernel of the homomorphism $Aut_{bialg}(H)\to Out_{alg}(H)$.
\end{prop}

In particular, $\pi_0(\St_{Aut^1_{\Tw}(H)}(\iota,1))$ is isomorphic to the kernel of $Out_{bialg}(H)\to Out_{alg}(H)$.

We finish this section with a description of the orbits of the Cat-group action of $\Aut_{\Tw}(H)$ on $\Tw(k,H)$ in
terms of twisted forms of the bialgebra $H$. Recall that a twist $F\in H^{\otimes 2}$ allows us to define a new
coproduct on $H$: $$\Delta^F(x) = F^{-1}\Delta(x)F.$$ We call the bialgebra $Tf(F) = (H,\Delta^F)$ an {\em $F$-twisted
form} of $H$ (or just a {\em twisted form}). Note that a gauge transformation of twists $F' (a\otimes a)= \Delta(a)F$
defines a homomorphism of bialgebras $(\ )^a:(H,\Delta^{F'})\to (H,\Delta^F)$. Indeed, $$\Delta^F(a^{-1}xa) =
F^{-1}\Delta(a)^{-1}\Delta(x)\Delta(a)F =$$ $$(a\otimes a)^{-1}(F')^{-1}\Delta(x)F'(a\otimes a) = (a\otimes
a)^{-1}\Delta^{F'}(x)(a\otimes a).$$ Thus we have a functor $Tf:\Tw(k,H)\to Bialg$ into the category $Bialg$ of
bialgebras and their automorphisms.
\begin{prop}
The functor $Tf:\Tw(k,H)\to Bialg$ is strictly constant on orbits of the action of $\Aut_{\Tw}(H)$ on $\Tw(k,H)$, i.e.
isomorphisms of twisted forms $f:(H,\Delta^F)\to (H,\Delta^{F'})$ are in 1-1 correspondence with twisted automorphisms
$(f,F''):(H,\Delta)\to (H,\Delta)$ such that $(f,F'')\circ(\iota,F) = (\iota,F')$.
\end{prop}
\begin{proof}
If $(f,F''):(H,\Delta)\to (H,\Delta)$ is a twisted automorphism such that $F''(f\otimes f)(F)= F'$ then $$(f\otimes
f)\Delta^F(x) = (f\otimes f)(F)^{-1}(f\otimes f)\Delta(x)(f\otimes f)(F) =$$ $$(F')^{-1}F''(f\otimes
f)\Delta(x)(F'')^{-1}F' = (F')^{-1}\Delta(f(x))F' = \Delta^{F'}(f(x))$$ so that $f$ is an isomorphism of bialgebras
$(H,\Delta^F)\to (H,\Delta^{F'})$. Conversely, for an isomorphism of bialgebras $f:(H,\Delta^F)\to (H,\Delta^{F'})$,
the element $F''= F'(f\otimes f)(F)^{-1}$ defines a structure of twisted isomorphism $(f,F''):(H,\Delta)\to
(H,\Delta)$: $$F''(f\otimes f)\Delta(x) = F'(f\otimes f)(F)^{-1}(f\otimes f)\Delta(x) =$$ $$F'(f\otimes
f)(F^{-1}\Delta(x)) = (F')^{-1}F'\Delta(f(x))F'(f\otimes f)(F)^{-1}.$$
\end{proof}

In particular, the orbits of the group action of $Aut_{\Tw}(H)$ on $Tw(k,H)$ are in 1-1 correspondence with isomorphism
classes of twisted forms of $H$.

\subsection{Action on triangular structures}
Here we refine the action of twisted automorphisms on triangular structures to the level of categories.

Recall that a {\em triangular structure} on a bialgebra $H$ is an invertible element $R\in H\otimes H$ (a {\em
universal R-matrix}) satisfying
\begin{equation}\label{conjr}
Rt\Delta (x) = \Delta (x)R \quad \forall x\in H,
\end{equation}
along with {\em triangle equations}: $$(I\otimes \Delta)(R) = R_{13}R_{12},$$ $$(\Delta\otimes I)(R) = R_{13}R_{23},$$
{\em normalisation}: $$(\varepsilon\otimes I)(R) = (I\otimes\varepsilon)(R) = 1,$$ and {\em unitarity} condition:
$$R_{21} = R^{-1}.$$ Here $R_{21} = (12)R$ is the transposition of tensor factors of $R\in H\otimes H$, $R_{12} =
R\otimes 1$, $R_{13} = (I\otimes (12))(R_{12})$ etc.

Note that any universal $R$-matrix on a cocommutative bialgebra $H$ satisfies the 2-cocycle condition: $$(R\otimes
1)(\Delta\otimes I)(R) = R_{12}R_{13}R_{23} = (I\otimes\Delta)(R)(1\otimes R).$$
\newline
Denote by $\Tr(H)$ the set of triangular structures on the bialgebra $H$.
\newline
It was observed by Drinfeld (see also \cite{ra}) that the map
\begin{equation}\label{suphom}
H^*\to H,\quad l\mapsto (l\otimes I)(R)
\end{equation}
is a homomorphism of algebras and anti-homomorphism of coalgebras for any triangular structure $R$. In particular, its
image is a (finite-dimensional) sub-bialgebra $H_R$ in $H$, $R$ belongs to $H_R^{\otimes 2}$ (the so-called {\em
minimal triangular} sub-bialgebra), and the map (\ref{suphom}) factors as follows $$H^*\to H^*_R\simeq H_R\to H,$$
where the first surjection $H^*\to H^*_R$ is dual to the last inclusion $H_R\to H$ and the isomorphism $H^*_R\simeq
H_R$ is self-dual. It is well known that the minimal triangular subalgebra for a cocommutative Hopf algebra possesses a
quite simple description.
\begin{prop}\label{minco}
For a cocommutative Hopf algebra $H$ the set $\Tr(H)$ of triangular structures is isomorphic to the set of pairs
$(A,b)$, where $A$ is a normal commutative cocommutative finite dimensional sub-bialgebra in $H$ and $b:A^*\to A$ is an
$H$-invariant isomorphism of bi-algebras.
\end{prop}
\begin{proof}
As a sub-biagebra of $H$ the minimal triangular sub-bialgebra $H_R$ must be cocommutative. It is commutative since
$H_R^*$ is a sub-bialgebra of $H^*$. Normality of $H_R$ and $H$-invariance of $b:H^*_R\simeq H_R$ are equivalent to the
condition (\ref{conjr}) for $R$.
\end{proof}

For a twisted automorphism $(f,F)$ and a triangular structure $R$ on $H$ define a {\em twisted} triangular structure:
$${R}^{(f,F)}= F_{21}(f\otimes f)(R)F^{-1}.$$ It is straightforward to verify that the properties of the $R$-matrix are
preserved. Moreover, gauge isomorphic twisted automorphisms act equally. Indeed, for $g(x) = af(x)a^{-1}$ and $G =
\Delta(a)F(a\otimes a)^{-1}$, $${R}^{(g,G)} = \Delta(a)F_{21}(a\otimes a)^{-1}(a\otimes a)(f\otimes f)(R)(a\otimes
a)^{-1}(a\otimes a)F^{-1}t\Delta(a)^{-1} =$$ $$\Delta(a){R}^{(f,F)}t\Delta(a)^{-1} ={R}^{(f,F)}.$$ Thus an action of
the group $Out_{\Tw}(H)$ on the set $\Tr(H)$ is defined.

Recall that the {\em Drinfeld element} of a triangular structure $R$ on a Hopf algebra $H$ is $u = \mu(I\otimes S)(R)$,
where $S:H\to H$ is the antipode and $\mu:H\otimes H\to H$ is the multiplication map. It was proven in \cite{dr1} that
$u$ (also see \cite{lu}) is a group-like element: $\Delta(u) = u\otimes u$. Moreover, if $H$ is cocommutative then $u$
is central and of order 2. Thus we have a map $u:Tr(H)\to (G(H)\cap Z(H))_2$ from the set of triangular structures to
the 2-torsion subgroup of group-like central elements. This map admits a section, which sends a group-like involution
$u$ to an $R$-matrix $$R_u = \frac{1}{2}(1\otimes 1 + 1\otimes u + u\otimes 1 - u\otimes u).$$

\begin{prop}
For a cocommutative Hopf algebra  the Drinfeld element map is constant on orbits of the action of $Out^1_{\Tw}(H)$ on
$Tr(H)$.
\end{prop}
\begin{proof}
It is quite straightforward to check that the Drinfeld element is constant on orbits. Indeed, it was proved in
\cite{dr1} that the map $\mu(I\otimes S):(H^{\otimes 2})^H\to Z(H)$ is a homomorphism of algebras. Thus $$\mu(I\otimes
S)(F_{21}RF^{-1}) =  \mu(I\otimes S)(F_{21})\mu(I\otimes S)(R)\mu(I\otimes S)(F)^{-1}$$ equals $ \mu(I\otimes S)(R)$
since $ \mu(I\otimes S)(F)=1$.
\end{proof}

Categorically triangular structures correspond to symmetric structures on the category of modules. Recall that a
monoidal category $\cG$ is {\em symmetric} if it is equipped with a collection of isomorphisms $c_{X,Y}:X\otimes Y\to
Y\otimes X$ natural in $X,Y\in\cG$ and satisfying the following axioms: $$c_{X,Y}c_{Y,X} = 1,\quad \mbox{symmetry},$$
{\em hexagon} axioms: $$c_{X,Y\otimes Z} = \phi_{Y,Z,X}(Y\otimes c_{X,Z})\phi_{Y,X,Z}^{-1}(c_{X,Y}\otimes
Z)\phi_{X,Y,Z},$$ $$c_{X\otimes Y,Z} = \phi_{Z,X,Y}^{-1}(c_{X,Z}\otimes Y)\phi_{X,Z,Y}(X\otimes
c_{Y,Z})\phi_{X,Y,Z}^{-1}.$$ Note that the last condition is redundant and included here for the sake of symmetry.
\begin{prop}
A triangular structure $R$ on a bialgebra $H$ defines a symmetric structure: $$c_{M,N}:M\otimes N\to N\otimes N,\quad
m\otimes n\mapsto R(n\otimes m)$$ on the category $H-Mod$
\end{prop}
\begin{proof}
The condition (\ref{conjr}) implies that $c_{M,N}$ is a morphism of $H$-modules: $$c_{M,N}(\Delta(x)(m\otimes
n)) = Rt\Delta(x)(n\otimes m) = \Delta(x)R(n\otimes m) = \Delta(x)c_{M,N}(m\otimes n).$$ Triangle equations are
equivalent to hexagon axioms. Normalisation for $R$ gives the conditions $c_{1,N} =I$, $c_{M,1} = I$. Unitarity for $R$
implies symmetry for $c$.
\end{proof}

Monoidal autoequivalences of a monoidal category act naturally on the set of symmetric structures of the category. For
a monoidal autoequivalence $F$ and a symmetry $c$ define the new symmetry $c^F$ by $$\xymatrix{ F(X)\otimes F(Y)
\ar[r]^{c^F_{F(X),F(Y)} }  & F(Y)\otimes F(X)\\ F(X\otimes Y) \ar[u]^{F_{X,Y}} \ar[r]^{F(c_{X,Y})} & F(Y)\otimes F(X)
\ar[u]^{F_{Y,X}} }$$ It is straightforward to see that this action corresponds to the action of twisted homomorphisms
on $R$-matrices.

\section{Twisted automorphisms of universal enveloping algebras}
In this part we look at twisted automorphisms of universal enveloping algebras over formal power series $k[[h]]$.

\subsection{Invariant twists of $U(\g)[[h]]$}\label{invtwi}
Let $F\in U(\g)[[h]]^\g$ be a twist. Expand it as a formal power series in $h$ with coefficients in $U(\g)$: $$F =
\sum_{i=0}^\infty F_ih^i.$$ Since invertible elements of a universal enveloping algebra over a field $k$ of
characteristic zero are trivial (scalars), the constant term $F_0$ of $F$ must be the identity. Let $X=F_l$ be the
first non-zero coefficient. The degree $l$ part of the 2-cocycle equation is the additive 2-cocycle condition for $X$:
$$1\otimes X + (I\otimes\Delta)(X) = X\otimes 1 + (\Delta\otimes I)(X).$$ Following Drinfeld \cite{dr} consider the
complex $(H^{\otimes *},\partial)$ with the differential $\partial:H^{\otimes n}\to H^{\otimes n+1}$ defined by
\begin{equation}\label{tangcoh}
\partial(X) = 1\otimes X +\sum_{i=1}^n(-1)^i(I^{\otimes i-1}\otimes\Delta\otimes I^{\otimes n-i-1})(X) + (-1)^{n+1}(X\otimes 1).
\end{equation}
The cohomology of this complex admits a simple description.
\begin{prop}
For a universal enveloping algebra $H=U(\g)$ the alternation map $Alt_n:H^{\otimes n}\to \Lambda^n H$ induces an
isomorphism of the n-th cohomology of the complex (\ref{tangcoh}) and $\Lambda^n\g$.
\end{prop}
\begin{proof}
Sketch of the proof (for details see \cite{dr}):
\newline
By the Poincare-Birkhoff-Witt theorem, the universal enveloping algebra $U(\g)$ is isomorphic as a coalgebra to the
symmetric algebra $S^*(\g)$. The complex (\ref{tangcoh}) for $H = S^*(\g)$ breaks into graded pieces:
\begin{equation}\label{grpiece}
S^n(\g)\to \oplus_{i_1+i_2=n}S^{i_1}(\g)\otimes S^{i_2}(\g)\to ...\to
\oplus_{i_1+...+i_s=n}\otimes_{j=1}^sS^{i_j}(\g)\to ...\to (\g)^{\otimes n}
\end{equation}
The degree $n$ piece is isomorphic to the cochain complex of the simplicial $n$-cube tensored (over symmetric group)
with $(\g)^{\otimes n}$.
\end{proof}

In particular, for an additive 2-cocycle $X\in U(\g)^{\otimes 2}$ there is $a\in U(\g)$ such that
\begin{equation}\label{cob}
X = \overline{X} + a\otimes 1 + 1\otimes a - \Delta(a),\quad \overline{X} = Alt_2(X) = \frac{1}{2}(X - X_{21}).
\end{equation}
Note that both $X$ and $\overline{X}$ are $\g$-invariant which makes $\partial(a) = a\otimes 1 + 1\otimes a -
\Delta(a)$ $\g$-invariant. The last implies that $a$ can be chosen to be $\g$-invariant (central). To see it we prove a
slightly more general statement.
\begin{lem}
For a universal enveloping algebra $H=U(\g)$ the alternation map $Alt_n:(H^{\otimes n})^H\to (\Lambda^n H)^H$ induces
an isomorphism of the n-th cohomology of the subcomplex of $H$-invariant elements of (\ref{tangcoh}) and the space of
$\g$-invariant skewsymmetric tensors $(\Lambda^n\g)^\g$.
\end{lem}
\begin{proof}
The coalgebra isomorphism between $U(\g)$ and $S^*(\g)$ is $\g$-invariant. The isomorphism between the degree $n$ piece
(\ref{tangcoh}) and the cochain complex of the simplicial $n$-cube tensored with $(\g)^{\otimes n}$ is natural in $\g$
and in particular $\g$-invariant.
\end{proof}

For the central $a$ satisfying (\ref{cob}) the exponent $exp(ah^l)$ defines a gauge transformation of invariant twists
$F\to F'$ where $F' = 1 + \overline{X}h^l + ... .$. Thus we can assume (up to a gauge transformation) that
$X\in(\Lambda^2\g)^\g$. Note that $\g$-invariance of $X$ implies $$[1\otimes X,(I\otimes\Delta)(X)] = [X\otimes
1,(\Delta\otimes I)(X)] = 0.$$ Hence the exponent $exp(Xh^l)$ is an invariant twist on $U(\g)[[h]]$:
$$(exp(Xh^l)\otimes 1)(\Delta\otimes I)(exp(Xh^l)) = exp((X\otimes 1 + (\Delta\otimes I)(X))h^l) =$$ $$exp((1\otimes X
+ (I\otimes\Delta)(X))h^l) = (1\otimes exp(Xh^l))(I\otimes\Delta)(exp(Xh^l)).$$ Writing $F$ as $exp(Xh^l)\circ F'$ we
will have at least the first $l$ components of $F'$ being zero. Iterating the argument we prove the following.
\begin{prop}
Any invariant twist on $U(\g)[[h]]$ is gauge isomorphic to a product $\prod_{i=1}^\infty exp(X_ih^i)$ where $X_i\in
(\Lambda^2\g)^\g$.
\end{prop}

Note that the components $X_i$ are defined uniquely by the twist $F$. Thus the set $\pi_0(\Aut^1_{\Tw}(U(\g)[[h]]))$ of
classes of invariant twists is isomorphic to $(\Lambda^2\g)^\g[[h]]$. Another way to see it is to use the logarithmic
map. Since, for an invariant twist $F$, the factors $1\otimes F,\ (I\otimes\Delta)(F)$ and $F\otimes 1,\ (\Delta\otimes
1)(F)$ of the 2-cocycle equation pairwise commute, the logarithm $log(F)$ is an additive 2-cocycle. Hence $F$ is gauge
isomorphic to $exp(X)$, where $X = Alt_2(log(F))\in (\Lambda^2\g)^\g[[h]]$.

To examine the group structure on $\pi_0(\Aut^1_{\Tw}(U(\g)[[h]]))$ we will use the Baker-Campbell-Hausdorff formula:
$$exp(X)exp(Y) = exp(X+Y)exp(A(X,Y)),$$ where $A(X,Y)$ is an element of the completion (with respect to the natural
grading) of the free Lie algebra on $X,Y$. Note that $$A(X,Y)= \frac{1}{2}[X,Y] + \mbox{higher terms}.$$ Now for
$X,Y\in (\Lambda^2\g)^\g[[h]]$ the commutator $[X,Y]$ is an additive $\g$-invariant 2-cocycle and is symmetric. Thus
there is a central $a(X,Y)\in Z(U(\g))[[h]]$ such that
\begin{equation}\label{cobcom}
[X,Y] = a(X,Y)\otimes 1 + 1\otimes a(X,Y) - \Delta(a(X,Y)).
\end{equation}
Note also
that any $Z\in (\Lambda^2\g)^\g[[h]]$ must commute with $[X,Y]$. Indeed, the first commutator in $$[[X,Y],Z] =
[a(X,Y)\otimes 1 + 1\otimes a(X,Y),Z] - [\Delta(a(X,Y)),Z]$$ is zero by centrality of $a(X,Y)$ while the second
vanishes because of $\g$-invariance of $Z$. In particular, the higher terms in $A(X,Y)$ are all zero and $A(X,Y) =
\partial(\frac{1}{2}a(X,Y))$. Thus the exponent $exp(\frac{1}{2}a(X,Y))$ is a gauge transformation between the
invariant twists $exp(X)exp(Y)$ and $exp(X+Y)$, which proves the following statement.
\begin{theo}
The group $\pi_0(\Aut^1_{\Tw}(U(\g)[[h]]))$ of classes of invariant twists is isomorphic to the addititve group
$(\Lambda^2\g)^\g[[h]]$.
\end{theo}

Moreover, we can calculate the associator class $$\phi\in H^3(\pi_0(\Aut^1_{\Tw}(U(\g)[[h]])),
\pi_1(\Aut^1_{\Tw}(U(\g)[[h]])))$$ of the Cat-group $\Aut^1_{\Tw}(U(\g)[[h]])$ of invariant twists. For our choice of
gauge transformation between the invariant twists $exp(X)exp(Y)$ and $exp(X+Y)$ the logarithm of the associator on
$exp(X),exp(Y),exp(Z)$ (as an element of $Z(\g)[[h]]\subset Z(U(\g))[[h]]$) equals
\begin{equation}\label{ass}
\frac{1}{2}(a(X,Y) + a(X+Y,Z) - a(Y,Z) - a(X,Y+Z)).
\end{equation}
Since both $\pi_i(\Aut^1_{\Tw}(U(\g)[[h]])), i=0,1$ are divisible abelian groups (vector spaces over $k$) with trivial
action of the first on the second, the cohomology group $$H^3(\pi_0(\Aut^1_{\Tw}(U(\g)[[h]])),
\pi_1(\Aut^1_{\Tw}(U(\g)[[h]])))\simeq H^3((\Lambda^2\g)^\g[[h]], Z(\g)[[h]])$$ is isomorphic to the group
$Hom(\Lambda^3((\Lambda^2\g)^\g[[h]]), Z(\g)[[h]])$ of skew-symmetric maps via the map which takes a group 3-cocycle
into its alternation. Clearly, the alternation of the (logarithm of the) associator (\ref{ass}) is zero. Thus the class
$\phi$ is trivial.

\begin{rem}Tangent Cat-Lie algebra of $\Aut^1_{\Tw}(U(\g))$.
\end{rem}
We can formalise the ground ring dependence of $\Aut^1_{\Tw}(U(\g))$ in the form of a Cat-group valued pseudo-functor
$$k\mapsto \Aut^1_{\Tw}(U_k(\g))$$ on the category of Artinian local commutative algebras. This point of view allows us
to define in a standard way the tangent Cat-Lie algebra (crossed module of Lie algebras) for $\Aut^1_{\Tw}(U(\g))$:
$$Z^2 \stackrel{\partial}{\leftarrow} C^1.$$ Here $Z^2$ is the Lie algebra (with respect to the commutator in
$U(\g)^{\otimes 2}$) of 2-cocycles of the subcomplex of $\g$-invariants of (\ref{tangcoh}) and $C^1$ is the abelian Lie
algebra of 1-cochains of the same subcomplex. The action of $Z^2$ on $C^1$ is trivial and the commutator
$[X,\partial(a)]$ is zero for any $X\in Z^2$ and $a\in C^1$ (thus fulfilling the axioms of a crossed module of Lie
algebras). Writing $[X,Y] = [Alt_2(X),Alt_2(Y)]$ as $\partial(a(X,Y))$ for $a(X,Y) = a(Alt_2(X),Alt_2(Y))$ as before we
can see that $a(X,[Y,Z]) = 0$ so the Jacobiator of the crossed module of Lie algebras is trivial.

\subsection{Separation for twisted automorphisms}
Here we examine twisted automorphisms $(f,F):H\to H$ where $H = U(\g)[[h]]$. The constant term (with respect to $h$) of
$F$ must be the identity. Thus by condition (\ref{conj}) the constant term of $f$ must be an automorphism of the
bialgebra $H$, hence must be induced by an automorphism of the Lie algebra $\g$. These allow us to assume without loss
of generality (up to an automorphism of $\g$) that $$f = I + \sum_{i=1}^\infty f_ih^i,\quad F = 1 + \sum_{i=1}^\infty
F_ih^i .$$ Let $X=F_l$ be the first non-zero coefficient. As before, the degree $l$ part of the 2-cocycle equation is
the additive 2-cocycle condition for $X$: $$1\otimes X + (I\otimes\Delta)(X) = X\otimes 1 + (\Delta\otimes I)(X).$$ As
before we write $$X = \overline{X} + a\otimes 1 + 1\otimes a - \Delta(a)$$ for $\overline{X} = Alt_2(X)$ and some $a\in
U(\g)$. The exponent $exp(ah^l)$ defines a gauge transformation of twisted automorphisms $(f,F)\to(f',F')$ where $F' =
1 + \overline{X}h^l + ... .$. Hence we can assume (up to a gauge transformation) that $X\in\Lambda^2\g$. Now the left
hand side of the degree $l$ part of the condition (\ref{conj}), namely, $$\sum_{i+j=l}(f_i\otimes f_j)\Delta(x) -
\Delta(f_l(x)) = [X,\Delta(x)]$$ is symmetric while the right hand side is anti-symmetric. That means both sides are
zero. In particular, $X\in\Lambda^2\g$ is $\g$-invariant and $exp(Xh^l)$ is a twist on $U(\g)[[h]]$. Writing $(f,F)$ as
$(1,exp(Xh^l))\circ(f,F')$ we will have at least the first $l$ components of $F'$ being zero. Proceeding like that
(using induction by the number of first successive zero components in $F$) we prove the following.
\begin{prop}
Any twisted automorphism of a universal enveloping algebra $U(\g)[[h]]$ is gauge isomorphic to a separated twisted
automorphism, i.e. a twisted automorphism of the form $(1,F)\circ(f,1)$ where $F$ is an invariant twist on $U(\g)[[h]]$
and $f$ is a bialgebra automorphism of $U(\g)[[h]]$.
\end{prop}

\subsection{The Cat-group $\Aut_{\Tw}(U(\g)[[h]])$}

Since any symmetric invariant twist is isomorphic to a trivial one, the group $$Aut_{\Tw}(U(\g)[[h]]) =
\pi_0(\Aut_{\Tw}(U(\g)[[h]]))\simeq Out_{bialg}(U(\g)[[h]])\ltimes (\Lambda^2\g)^\g[[h]]$$ is the crossed product of
the group of outer bialgebra automorphisms of $U(\g)[[h]]$ and the group of gauge classes of invariant twists. Any
bialgebra automorphism of a universal enveloping algebra is induced by a Lie algebra automorphism. Thus the group
$Out_{bialg}(U(\g)[[h]])$ in its turn is the crossed product $$Aut(\g)\ltimes (1+hOutDer(\g)[[h]])$$ of the group of
automorphisms of the Lie algebra $\g$ and the exponent of the Lie algebra $hOutDer(\g)[[h]]$ of outer derivations of
$\g[[h]]$ of degree $\geq 1$. The action of the subgroup $Aut(\g)$ on the group of gauge classes of invariant twists
$(\Lambda^2\g)^\g[[h]]$ is $$f(F) = (f\otimes f)(F).$$ The action of the degree $\geq 1$ part $(1+hOutDer(\g)[[h]])$ is
induced by the action of the Lie algebra of derivations $Der(\g)$ on the space $(\Lambda^2(\g))^\g$: $$dX = (d\otimes I
+ I\otimes d)(X).$$ It is straightforward to see that inner derivations act trivially (this is equivalent to
$\g$-invariance). To see that $\g$-invariance is preserved by this action we need to verify that $d_xd(X) =0$ for any
$x\in\g$. Here $d_x(y) = [x,y]$ is the inner derivation corresponding to $x$. Since $d_xd = dd_x + d_{d(x)}$ we have
$d_xdX = dd_x(X) + d_{d(x)}X = 0$.

Note that $\pi_1(\Aut_{\Tw}(U(\g)[[h]])) = \pi_1(\Aut^1_{\Tw}(U(\g)[[h]])) = Z(\g)[[h]]$.
\begin{prop}
The associator $$\phi\in H^3(Aut_{\Tw}(U(\g)[[h]]),Z(U(\g)[[h]]))$$ of the Cat-group $\Aut_{\Tw}(U(\g)[[h]])$ is the
image of the associator $$\psi\in H^3(Out_{bialg}(U(\g)[[h]]),Z(U(\g)[[h]]))$$ of the Cat-group
$\Aut_{bialg}(U(\g)[[h]])$ under the homomorphism of groups $Out_{bialg}(U(\g)[[h]])\to Aut_{\Tw}(U(\g)[[h]])$.
\end{prop}
\begin{proof}
We need to check that the associator is trivial if at least one of the arguments belongs to the subgroup
$(\Lambda^2\g)^\g[[h]]$. We have seen in section \ref{invtwi} that it is trivial if all three arguments are from
$(\Lambda^2\g)^\g[[h]]$. In section \ref{geomdes} we will construct the solution $a(X,Y)$ of (\ref{cobcom}) such that
$a(g(X),g(Y)) = a(X,Y)$ for any automorphism $g\in Aut(\g)$ (and $a(dX,Y) + a(X,dY) = 0$ for any derivation $d\in
Der(\g)$) thus covering the case when two of the arguments of the associator belong to $(\Lambda^2\g)^\g[[h]]$.
Finally, the fact $Out(\g)$ and $OutDer(\g)$ act on $(\Lambda^2\g)^\g$ guarantees that the associator is trivial when
one of the arguments belongs to the subgroup $(\Lambda^2\g)^\g[[h]]$.
\end{proof}

\subsection{Twists on $U(\g)[[h]]$}
Recall that $\Tw(k[[h]],U(\g)[[h]])$ is a complicated notation for the groupoid of twists on $U(\g)[[h]]$.

Writing the twist as a formal power series $F = \sum_{i=0}^\infty f_i h^i$, we get, for the degree $n$ part of the
2-cocycle equation: $$\sum_{i+j=n,\ i,j>n}((1\otimes f_i)(I\otimes\Delta)(f_j) - (f_i\otimes 1)(\Delta\otimes I)(f_j))
= 0.$$ The degree 1 part is simply $d(f_1) = 0$. Thus $f_1$ is an additive 2-cocycle and $r = Alt_2(f_1)$ belongs to
$\Lambda^2(\g)$. Up to a gauge isomorphism we can assume that $f_1 = r$. The degree 2 part reads as
\begin{equation}\label{deg2}
(1\otimes f_1)(I\otimes\Delta)(f_1) - (f_1\otimes 1)(\Delta\otimes I)(f_1) = d(f_2).
\end{equation}
The left hand side is in general a 2-cocycle (this can be checked directly). Thus the alternation of the left hand side
is an element of $\Lambda^3(\g)$. This element can be written explicitly if we assume as before that $f_1 = r$ is
bi-primitive and skew-symmetric. Indeed, for such a choice, the left hand side of (\ref{deg2}) has the form
$$r_{23}(r_{13} + r_{12}) - r_{12}(r_{13} + r_{23}).$$ Note that $$Alt_3(r_{23}r_{13}) = \frac{1}{6}(r_{23}r_{13} +
r_{31}r_{21} + r_{12}r_{32} - r_{13}r_{23} - r_{32}r_{12} - r_{21}r_{31}),$$ which for a skew-symmetric $r$ equals
$$\frac{1}{6}([r_{23},r_{13}] + [r_{23},r_{12}] + [r_{13},r_{12}]).$$ Note also that $$Alt_3(r_{23}r_{13}) =
Alt_3(r_{23}r_{12}) = - Alt_3(r_{12}r_{13}) = - Alt_3(r_{12}r_{23}).$$ Thus $$Alt_3(r_{23}(r_{13} + r_{12}) -
r_{12}(r_{13} + r_{23})) = $$ $$\frac{4}{6}([r_{23},r_{13}] + [r_{23},r_{12}] + [r_{13},r_{12}]).$$ Together with
(\ref{deg2}) it means that
\begin{equation}\label{cybe}
[r_{23},r_{13}] + [r_{23},r_{12}] + [r_{13},r_{12}] = 0.
\end{equation}
This equation is known as the {\em classical Yang-Baxter} equation ({\em CYBE}). Denote by $CYB(\g)$ the set of
solutions (in $\Lambda^2(\g)$) of the classical Yang-Baxter equation. Thus we have the following.
\begin{prop}
The above defines a map
\begin{equation}\label{qcl}
\pi_0(\Tw(k[[h]],U(\g)[[h]]))\to CYB(\g)
\end{equation}
from the set of gauge isomorphism classes to the set of solutions to CYBE.
\end{prop}
A section to the map (\ref{qcl}) was constructed in \cite{dr0} (see also \cite{es}). In particular, the map (\ref{qcl})
is surjective.

The map (\ref{qcl}) can be extended to a map of groupoids in the following way. For a gauge automorphism $a\in
U(\g)[[h]]$ of a twist $F\in \Tw(k[[h]],U(\g)[[h]])$ the condition $\Delta(a)F = F(a\otimes a)$ expanded in $h$ gives
$$\sum_{i+j=n}\Delta(a_i)f_j = \sum_{i+j=n}f_j(a_i\otimes 1 + 1\otimes a_i).$$ Since $a_0 = f_0 = 1$, the degree 1 part
is $\Delta(x) + r = r + x\otimes 1 + 1\otimes x$, where $x=a_1$ and $r = f_1$. Thus $x$ belongs to $\g$. After
cancellation the skew-symmetric degree 2 part reads $\Delta(x)r = r(x\otimes 1 + 1\otimes x)$ or
\begin{equation}\label{centr}
[r,x\otimes 1 + 1\otimes x] = 0.
\end{equation}
Define the {\em centraliser} $C_\g(r)$ of $r$ as the Lie subalgebra in $\g$ of those $x$ which satisfy (\ref{centr}).
Denote by $\CYB(\g)$ a disconnected groupoid with the set of (classes of) objects $CYB(\g)$ and with the abelian
automorphism groups $Aut_{\CYB(\g)}(r) = C_\g(r)$ where $r\in \CYB(\g)$. Then the map (\ref{qcl}) lifts to a functor of
groupoids $$\Tw(k[[h]],U(\g)[[h]])\to \CYB(\g).$$

The action of the cat-group $\Aut^1_\Tw(U(\g)[[h]])$ on $\Tw(k[[h]],U(\g)[[h]])$ corresponds to the action of
$(\Lambda^2(\g))^\g$ on $CYB(\g)$ given by addition: for $X\in (\Lambda^2(\g))^\g$ and $r\in CYB(\g)$ the sum $X + r$
belongs to $CYB(\g)$. The action of $Aut(\g[[h]])$ on $CYB(\g)$ boils down to the group action of $Aut(\g)$.

\begin{rem}
\end{rem}
It seems that the construction of \cite{dr0} (as well as \cite{es}) can be extended to a bijection $CYB_h(\g)\to
\pi_0(\Tw(k[[h]],U(\g)[[h]]))$. Here we understand $CYB_h(\g)$ as the set of solutions to CYBE in $\Lambda^2(\g)[[h]]$.
$\g$-equivariance of the construction from \cite{dr0} would imply that the map $\CYB_h(\g)\to \Tw(k[[h]],U(\g)[[h]])$
is a functor if we define automorphism groups in $\CYB_h(\g)$ to be exponents of Lie algebras $C_\g(r)[[h]]$.

\subsection{Geometric description}\label{geomdes}

For $X\in\Lambda^2(\g)$ define its {\em support subspace} to be $\ga(X) = \{(l\otimes I)(X), l\in\g^*\}\subset\g$. Note
that $X$ belongs to $\Lambda^2\ga(X)$ and defines a linear isomorphism $$\ga(X)^*\to\ga(X),\ l\mapsto (l\otimes)(X).$$
Thus $X$ is the {\em Casimir} element for the symplectic form $b=b(X)$ on $\ga(X)$: $$b((l\otimes I)(X),(l'\otimes
I)(X)) = (l\otimes l')(X)$$ $$(I\otimes b)(X\otimes I) = (b\otimes I)(I\otimes X) = I:\ga(X)\to\ga(X).$$

The following geometric characterisation of solutions to CYBE was obtained by Drinfeld \cite{dr0}.
\begin{prop}
The support $\s = \{(I\otimes l)(r),\ l\in\g^*\}$ of a solution $r\in\Lambda^2(\g)$ to CYBE is a Lie subalgebra of
$\g$. The symplectic form $b$ on $\s$ is a Lie 2-cocycle.
\end{prop}
\begin{proof}
Applying $(I\otimes l\otimes m)$ to CYBE we get $$[(I\otimes l)(r),(I\otimes m)(r)] + (I\otimes l')(r) + (I\otimes
m')(r) = 0,$$ where $l'(x) = l([x,(I\otimes m)(r)])$ and $m'(x) = m([x,(l\otimes I)(r)]$. Thus $\s\subset\g$ is a Lie
subalgebra.

To see that the form $b$ is a 2-cocycle note that $$b([(I\otimes l)(r),(I\otimes m)(r)],(I\otimes n)(r)) = (n\otimes
m\otimes l)([r_{12},r_{23}] - [r_{13},r_{23}).$$ Thus the alternation (in $l,m,n$) of the left hand side is $(n\otimes
m\otimes l)$ applied to a multiple of CYBE.
\end{proof}

In terms of $\s$ and $b$ the centraliser $C_\g(r)$ is the stabiliser $St_{C_\g(\s)}(b)$ of the form $b$ in the
centraliser $C_\g(\s)$ of the Lie subalgebra $\s$ in $\g$.

For $\g$-invariant $X\in\Lambda^2(\g)$, its support $\ga(X)$ is an ideal in $\g$ and the form $b(X)$ is $\g$-invariant.
\begin{lem}
For any $\g$-invariant $X\in\Lambda^2(\g)^{\g}$,
\begin{equation}\label{com}
[X_{13},X_{23}] = 0.
\end{equation}
\end{lem}
\begin{proof}
Since $X_{23}$ is $ad(\g)$-invariant $[X_{13},X_{23}] = - [X_{12},X_{23}]$ which equals $[X_{23},X_{12}]$. Again, by
$ad(\g)$-invariance of $X_{12}$, we have $[X_{23},X_{12}] = - [X_{13},X_{12}]$ which coincides with $[X_{13},X_{21}]$
since $X_{12} = - X_{21}$. Now, by $ad(\g)$-invariance of $X_{13}$, $[X_{13},X_{21}] = - [X_{13},X_{23}]$ which finally
implies
\begin{equation}
[X_{13},X_{23}] = - [X_{13},X_{23}].
\tag*{\qedhere}
\end{equation}
\end{proof}

It follows from the relation (\ref{com}) that $\ga(X)$ is an abelian ideal for any $X\in\Lambda^2(\g)^{\g}$:
$$[(l\otimes I)(X),(l'\otimes I)(X)] = [(l\otimes l'\otimes I)([X_{13},X_{23}]) = 0.$$ Thus we assign an abelian ideal
with an $\g$-invariant symplectic form to any element of $\Lambda^2(\g)^{\g}$. Conversely, for such a pair $(\ga,b)$
the Casimir element $X_b$ of $b$ obviously belongs to $\Lambda^2(\g)^{\g}$. Thus we have the following.
\begin{prop}\label{inv}
The support construction establishes a bijection between $\Lambda^2(\g)^{\g}$ and the set of pairs $(\ga,b)$, where
$\ga\subset \g$ is an abelian ideal and $b$ is a $\g$-invariant symplectic form of $\ga$
\end{prop}

For a Lie algebra $\g$ denote by $nil(\g)$ its {\em nilradical}, i.e. the maximal nilpotent ideal ( = sum of all
nilpotent ideals).
\begin{cor}
For any Lie algebra $\g$, $$\Lambda^2(\g)^{\g} = \Lambda^2(nil(\g))^{\g}.$$ In particular $\Lambda^2(\g)^{\g} = 0$ if
$nil(\g) = 0$.
\end{cor}
\begin{proof}
Any abelian ideal is obviously nilpotent. Thus it must be contained in $nil(\g)$.
\end{proof}

Now we describe vector space structure of $\Lambda^2(\g)$ in terms of pairs $(\ga,b)$. Obviously, multiplication by
scalars is given by the following rule: $$c(\ga,b) = (\ga,cb),\ \mbox{for}\ c\in k\setminus\{ 0\}.$$ The geometric
presentation for the addition is more involved.

For $X_1,X_2\in\Lambda^2(\g)$, the direct sum of their supports $\ga_1\oplus \ga_2$ is equipped with the symplectic
form $b_1\oplus b_2$. Denote by $$(\ga_1\cap \ga_2)^\perp = \{(u_1,u_2)\in \ga_1\oplus \ga_2: b_1(u_1,x) = b_2(u_2,x)\
\forall x\in \ga_1\cap \ga_2\}$$ the orthogonal complement of the anti-diagonal image of $\ga_1\cap \ga_2$ in
$\ga_1\oplus \ga_2$. Denote by $K$ the intersection $(\ga_1\cap \ga_2)\cap (\ga_1\cap \ga_2)^\perp$ which coincides
with the kernel $ker(b_1|_{\ga_1\cap \ga_2} - b_2|_{\ga_1\cap \ga_2})$ of the difference of the symplectic forms $b_i$
restricted to $\ga_1\cap \ga_2$. The kernel of the surjection $\ga_1\oplus \ga_2\to \ga_1 + \ga_2\subset \g$ coincides
with the anti-diagonal image of $\ga_1\cap \ga_2$ and the short exact sequence $$\ga_1\cap \ga_2\to \ga_1\oplus
\ga_2\to \ga_1 + \ga_2$$ extends to a commutative diagram with short exact rows and columns:
\begin{equation}\label{exdiag}
\xymatrix{ K \ar[r]
\ar[d] & \ga_1\cap \ga_2 \ar[r] \ar[d] & ((\ga_1\cap \ga_2)/K)^* \ar[d] \\ (\ga_1\cap \ga_2)^\perp \ar[r] \ar[d] &
\ga_1\oplus \ga_2 \ar[r] \ar[d] & (\ga_1\cap \ga_2)^* \ar[d] \\ \ga \ar[r] & \ga_1 + \ga_2 \ar[r] & K^* }
\end{equation}
We can use the bottom row to define a subspace $\ga\subset \g$ as the kernel of the map $\ga_1 + \ga_2\to K^*$ induced
by the map $\ga_1\oplus \ga_2\to (\ga_1\cap \ga_2)^*$: $$(u_1,u_2)\mapsto (x\mapsto b_1(u_1,x) - b_2(u_2,x))\ x\in
\ga_1\cap \ga_2\}.$$ Then the left column allows us to define a symplectic form on $\ga$. Indeed, the kernel of the
restriction of $b_1\oplus b_2$ to $(\ga_1\cap \ga_2)^\perp$ is $K$. Thus $b_1\oplus b_2$ induces a non-degenerate
skew-symmetric bilinear form $b$ on $(\ga_1\cap \ga_2)^\perp/K =\ga$.

\begin{prop}
The subspace $\ga\subset \g$ is the support of the sum $X_1 + X_2\in\Lambda^2(\g)$. The Casimir element of the
symplectic form $b$ coincides with $X_1+X_2$.
\end{prop}
\begin{proof}
First we need to verify that the support of $X_1+X_2$ lies in $\ga$. For that purpose we need a more explicit
description of the map $\ga_1 + \ga_2\to K^*$. For $(l_1\otimes I)(X_1) + (l_2\otimes I)(X_2)\in \ga_1+\ga_2$ it gives
a linear function $K\to k$: if we write an element $x\in K$ as $x = (m_1\otimes I)(X_1) = (m_2\otimes I)(X_2)$ the
value of this function on $x$ is $$(l_1\otimes m_1)(X_1) - (l_2\otimes m_2)(X_2).$$ Note that this expression is zero
for any $l_i\in \g^*$ such that $(l_1\otimes I)(X_1) = (l_2\otimes I)(X_2)$ so it is well defined on $\ga_1 + \ga_2$
(does not depend on the presentation $(l_1\otimes I)(X_1) + (l_2\otimes I)(X_2)$). The function associated to an
element $(l\otimes I)(X_1) + (l\otimes I)(X_2)$ of the support of $X_1+X_2$ is clearly zero. Thus the support of
$X_1+X_2$ belongs to $\ga$.

Now it suffices to check that $(b\otimes I)(I\otimes(X_1 + X_2))$ is the identity on $\ga$, or that $(b\otimes
l)(I\otimes(X_1 + X_2)) = l$ for any linear function $l$ on $\ga$. The bilinear form $b$ on $\ga$ assigns to
$(l_1\otimes I)(X_1) + (l_2\otimes I)(X_2), (m_1\otimes I)(X_1) + (m_2\otimes I)(X_2)\in \ga$ the number $$(l_1\otimes
m_1)(X_1) - (l_2\otimes m_2)(X_2).$$ In particular, for $x = (l_1\otimes I)(X_1) + (l_2\otimes I)(X_2)\in \ga$,
$$(b\otimes l)(x\otimes(X_1 + X_2)) = b(x,(I\otimes l)(X_1) + (I\otimes l)(X_2)) = $$
\begin{equation}
(l_1\otimes l)(X_1) + (l_2\otimes l)(X_2) = l(x).
\tag*{\qedhere}
\end{equation}
\end{proof}

In particular, if both $X_1,X_2$ are $\g$-invariant we get a geometric description of the addition on
$\Lambda^2(\g)^\g$.
\begin{rem}
\end{rem}
According to proposition \ref{inv}, the support subspace $\ga$ of $X_1 + X_2$ must be an abelian ideal. The fact that
$\ga$ is an ideal follows from the $\g$-equivariance of the construction for $\ga$, while the abelian property can be
checked by direct computation. First note that the commutant $[\ga_1,\ga_2]$ lies in $K$. As a consequence, the map
$\ga_1 + \ga_2\to K^*$ is a homomorphism of Lie algebras (with abelian structure on $K^*$). Thus $\ga$ is a Lie
subalgebra. To see that the commutant $[\ga,\ga]$ is zero, it is enough to check that $b_1([u_1+u_2,v_1+v_2],y) = 0$
for any $u_1+u_2,v_1+v_2\in \ga$ and any $y\in\ga_1$. Writing $$[u_1+u_2,v_1+v_2] = [u_1,v_2] + [u_2,v_1] = [u_1,v_2] -
[v_1,u_2],$$ we need to verify that $b_1([u_1,v_2],y) = b_1([v_1,u_2],y)$. Indeed, by $\g$-invariance of $b_i$ and the
defining relations for $u_1+u_2,v_1+v_2$ (together with $[\ga_1,\ga_2]\subset K$), we have the chain of equalities:
$$b_1([u_1,v_2],y) = - b_1(u_1,[y,v_2]) = - b_2(u_2,[y,v_2]) = $$ $$b_2([y,u_2],v_2) = b_1([y,u_2],v_1) = -
b_1(y,[v_1,u_2]) = b_1([v_1,u_2],y).$$

Note also that the arrows of the diagram (\ref{exdiag}) are homomorphisms of vector spaces and not of Lie algebras. To
turn it into a diagram of Lie algebras it suffices to introduce an appropriate Lie algebra structure on the direct sum
$\ga_1\oplus \ga_2$. Denote by $\ga_1\bowtie \ga_2$ the Lie algebra of pairs $(x_1,x_2)$, $x_i\in\ga_i$ with the
bracket: $$[(x_1,x_2),(y_1,y_2)] = \frac{1}{2}([x_2,y_1]+[x_1,y_2],[x_2,y_1]+[x_1,y_2]).$$ Then the maps of the diagram
\begin{equation}\label{exdiagm}
\xymatrix{ K \ar[r] \ar[d] & \ga_1\cap \ga_2 \ar[r] \ar[d] & ((\ga_1\cap \ga_2)/K)^* \ar[d] \\
(\ga_1\cap \ga_2)^\perp \ar[r] \ar[d] &  \ga_1\bowtie \ga_2 \ar[r] \ar[d] & (\ga_1\cap \ga_2)^* \ar[d] \\ \ga \ar[r] &
\ga_1 + \ga_2 \ar[r] & K^* }
\end{equation}
become homomorphisms of Lie algebras (one should think of objects in the right column as
abelian Lie algebras).

When $X_1$ is $\g$-invariant and $X_2$ is a solution to CYBE the sum $X_1 + X_2$ is also a solution to the CYBE. In
this case the support space $\ga$ is still a Lie algebra. Again the diagram (\ref{exdiag}) can be made into a diagram
of Lie algebras. One needs to think of $\ga_1\bowtie \ga_2$ as the Lie algebra with the bracket:
$$[(x_1,x_2),(y_1,y_2)] = \frac{1}{2}([x_2,y_1]+[x_1,y_2],[x_2,y_1]+[x_1,y_2]+2[x_2,y_2]).$$ In other words, the
diagram (\ref{exdiagm}) defines an action of the set of abelian ideals with invariant symplectic forms (which coincides
with the vector space $(\Lambda^2\g)^\g$) on the set of subalgebras with non-degenerate 2-cocycles (the set $CYB(\g)$).

Now we give a geometric description of the map
\begin{equation}\label{3vect}
(\Lambda^2\g)^\g\otimes (\Lambda^2\g)^\g\to (S^3\g)^\g,
\end{equation}
which for $X_1,X_2$ gives a coboundary $a$ for the commutator $$[X_1,X_2] = a\otimes 1 + 1\otimes a - \Delta(a).$$ Here
we identify $(S^*\g)^\g$ with the centre $Z(U(\g)) = U(\g)^\g$ of the universal enveloping algebra by means of the
Kirillov-Duflo isomorphism.

For two abelian ideals $\ga_1,\ga_2\subset \g$ denote by $\gb = [\ga_1,\ga_2]$ their commutant. Note that orthogonal
complements $\gb^\perp_{b_i}$ with respect to invariant symplectic forms $b_i$ on $\ga_i$ have the following
commutation property:
\begin{equation}\label{commut}
[\gb^\perp_{b_1},\ga_2] = [\ga_1,\gb^\perp_{b_2}] = 0.
\end{equation}
Indeed, for $x\in \gb^\perp_{b_1}, y\in\ga_2$, $$b_1([x,y],z) = - b_1(x,[z,y]) = 0,\quad \forall z\in \ga_1.$$ Hence
$[x,y] = 0$. Similarly for $[\ga_1,\gb^\perp_{b_2}]$. Now define a map $\gb^*\otimes\gb^*\to \gb$ by sending
$l_1\otimes l_2$ into $[x_1,x_2]$, where $x_i\in\ga_i$ are defined by $$b_i(x_i,u) = l_i(u),\quad \forall u\in\gb.$$
The elements $x_i$ are defined up to $\gb^\perp_{b_i}$ so in view of the relations (\ref{commut}) the commutator
$[x_1,x_2]$ is well defined. Obviously, the map $\gb^*\otimes\gb^*\to \gb$ corresponds to a 3-vector $c\in \gb^{\otimes
3}$, which is $\g$-invariant by the construction. To see that this 3-vector is symmetric, think of it as a map
$$(\gb^*)^{\otimes 3}\to k,\quad l_1\otimes l_2\otimes l_3\mapsto l_3([x_1,x_2]).$$ Find $y_i\in\ga_i$ such that
$$b_i(y_i,u) = l_3(u),\quad \forall u\in\gb.$$ Then $$l_3([x_1,x_2]) = b_1(y_1,[x_1,x_2]) = - b_1([y_1,x_2],x_1) =
l_1([y_1,x_2]),$$ which in terms of the 3-vector $c$ means that $c_{13} = c$. Similarly $$l_3([x_1,x_2]) =
b_2(y_2,[x_1,x_2]) = - b_2([x_1,y_2],x_2) = l_2([x_1,y_2]),$$ which means that $c_{23} = c$.

\begin{lem}
The element $a$ corresponding to the invariant $c\in S^3(\g)^\g$ via the isomorphism $S^*(\g)^\g\to Z(U(\g))$ satisfies
the equation $$[X_1,X_2] = a\otimes 1 + 1\otimes a - \Delta(a).$$
\end{lem}
\begin{proof}
Since $\gb = [\ga_1,\ga_2]$ is isotropic with respect to the forms $b_i$ we can choose subspaces $L_i\subset \ga_i$,
which are Lagrangian with respect to $b_i$ respectively and contain $\gb$. Write $\ga_i = L_i\oplus L_i^*$. A choice of
bases $L_1 = <e_i>,\ L_2 = <f_j>$ allows us to write $$X_1 = \sum_i e_i\wedge e^i,\quad X_2 = \sum_j f_j\wedge f^j.$$
Here $e^i$ ($f^j$) is the dual basis in $L_1^*$ (respectively $L_2^*$). By the choice, $[L_1,L_2]=0$ so $$[X_1,X_2] =
\sum_{i,j}[e_i\wedge e^i,f_j\wedge f^j] = \sum_{i,j}e_ie_j\odot [e^i,f^j].$$ Here $\odot$ is the symmetric product
$x\odot y = x\otimes y + y\otimes x$. Since the commutator pairing $[\ ,\ ]:L_1^*\otimes L^*_2\to \gb$ factors through
$c:\gb^*\otimes\gb^*\to\gb$ we can write $$a = \sum_{i,j}e_if_j[e^i,f^j] = \sum_{s,t}x_sx_tc(x^s,x^t),$$ where $x_s$ is
a basis of $\gb$. Finally note that the element $c$ belongs to $S^*(\gb)^\g = U(\gb)^\g\subset Z(U(\g))$ so we do not
need to worry about the particular choice of isomorphism $S^*(\g)^\g\to Z(U(\g))$.
\end{proof}

\begin{rem}
\end{rem}
It follows from the invariance of the construction of $c$ that $c(g(X_1),g(X_2)) = c(X_1,X_2)$ for an automorphism $g$
of the Lie algebra $\g$.

We can use a symmetric 3-vector $c\in \gb^{\otimes 3}$ on a vector space $\gb$ to define a structure of (meta-abelian)
Lie algebra $\g(\gb,c)$ on the vector space $\gb^*\oplus \gb^*\oplus \gb$: $$[(l_1,m_1,x_1),(l_2,m_2,x_2)] =
(0,0,(l_1\otimes m_2\otimes I - l_2\otimes m_1\otimes I)(c)).$$ The subspaces $$\ga_1 = \gb^*\oplus 0\oplus \gb,\quad
\ga_2 = 0\oplus \gb^*\oplus \gb$$ are abelian ideals in $\g(\gb,c)$. The symplectic forms $b_i$ on $\ga_i$
$$b_1((l_1,0,x_1),(l_2,0,x_2)) = l_1(x_2) - l_2(x_1),$$ $$b_2((0,m_1,x_1),(0,m_2,x_2)) = m_1(x_2) - m_2(x_1)$$ are
$\g(\gb,c)$-invariant. Indeed, by $c_{13} = c$ $$b_1([(0,m,x),(l_1,0,x_1)],(l_2,0,x_2)) = b_1(-(0,0,(l_1\otimes m\otimes
I)(c)),(l_2,0,x_2)) =$$ $=(l_1\otimes m\otimes l_2)(c),$ which coincides with $$b_1((l_1,0,x_1),[(l_2,0,x_2),(0,m,x)]) =
b_1((l_1,0,x_1),(0,0,(l_2\otimes m\otimes I)(c))) = (l_2\otimes m\otimes l_1)(c).$$ Similarly, by $c_{23} = c$
$$b_2([(0,m_1,x_1),(l,0,x)],(0,m_2,x_2)) = b_1(-(0,0,(l\otimes m_1\otimes I)(c)),(0,m_2,x_2)) =$$ $=(l\otimes m_1\otimes
m_2)(c)$, which coincides with $$b_2((0,m_1,x_1),[(l,0,x),(0,m_2,x_2)]) = b_2((0,m_1,x_1),(0,0,(l\otimes m_2\otimes I)(c))) =$$
$=(l\otimes m_2\otimes m_1)(c).$ Clearly, the 3-vector corresponding to the pair $(\ga_i,b_i)$ via (\ref{3vect}) is $c$.

The above construction is universal in the following sense. For a pair $(\ga_i,b_i)$ of abelian ideals with symplectic
invariant forms, the sum $\ga_1 + \g_2$ (which is a meta-abelian ideal) maps on to $\g(\gb,c)$, where $\gb =
[\ga_1,\ga_2]$ and $c$ is the 3-vector defined by (\ref{3vect}). Note that $\gb$ is isotropic in both $\ga_i$ so the
maps $\ga_i\to\gb^*$ defined by the forms $b_i$ factor through $\ga_i/\gb$. Thus we have a diagram $$\xymatrix{\gb
\ar[d] \ar[r] & \ga_1 + \ga_2 \ar[d] \ar[r] & \ga_1/\gb\oplus \ga_2/\gb \ar[d] \\ \gb \ar[r] & \g(\gb,c) \ar[r] &
\gb^*\oplus \gb^*}$$ The middle vertical map is a homomorphism of Lie algebras since the right vertical map is
compatible with commutator pairings $\ga_1/\gb\otimes \ga_2\to\gb$ (induced from $\ga_1 + \g_2$) and
$\gb^*\otimes\gb^*\to\gb$ (in $\g(\gb,c)$).

\begin{exam}\label{heisen}Heisenberg algebra
\end{exam}
Let $(V,b)$ be a symplectic vector space and $\g = {\cal H}(V,b) = V\oplus \langle c\rangle$ its Heisenberg Lie algebra with the central generator
$c$. Then the map $v\mapsto v\wedge c = v\otimes c - c\otimes v$ is an isomorphism $V\to (\Lambda^2\g)^\g$. Clearly
$v\wedge c$ is $\g$-invariant: $$[u\otimes 1 + 1\otimes u,v\wedge c] = [u,v]\wedge c = b(u,v)c\wedge c = 0.$$ To see
that there are no other $\g$-invariant elements in $\Lambda^2\g$ note that $\Lambda^2\g =
\Lambda^2V\oplus V$ with the $\g$-action on the first component $V\otimes\Lambda^2V\to V$ being induced by the bilinear
form $b$ and hence having no invariants.

The commutator (in $U(\g)^{\otimes 2}$) of two elements from $(\Lambda^2\g)^\g$ $$[v\wedge c,u\wedge c] =
b(v,u)(c\otimes c^2 + c^2\otimes c)$$ is a coboundary $a\otimes 1 + 1\otimes a - \Delta(a)$. For example, we can choose
$a:\Lambda^2((\Lambda^2\g)^\g)\to U(\g)$ to be $$a(v\wedge c,u\wedge c) = \frac{b(u,v)}{3}c^3.$$

The subalgebras in ${\cal H}(V,b)$ break into two classes depending on whether or not they contain the centre. A
subalgebra of the first type has the form $U\oplus <c>$ for an arbitrary subspace $U\subset V$, while the second type
is simply an {\em isotropic} subspace $U$ of $V$ (isotropic means that the restriction of the symplectic form $b$ to
$U$ is zero). A non-degenerate 2-cocycle on a subalgebra of the second type is a symplectic form. For a subalgebra of
the first type $U\oplus <c>$ with a non-degenerate 2-cocycle $\beta$ denote by $U'$ the orthogonal complement (with
respect to $\beta$) of $c$ in $U$. Clearly the restriction of $\beta$ on $U'$ is non-degenerate. Moreover, the
2-cocycle condition implies that $U'$ is isotropic with respect to $b$. This corresponds to the fact that any $X\in
CYB(\g)$ for $\g = {\cal H}(V,b)$ can be uniquely written as $Y + v\wedge c$ for $Y\in CYB(\g)$ supported on a
subalgebra of the second type. In particular, $(\Lambda^2\g)^\g$ acts transitively on $CYB(\g)$ with orbits
corresponding to subalgebras of the second type with symplectic forms on them.

\section{Crossed products with respect to actions by twisted automorphisms}
In this section we construct crossed product of a bialgebra with a group acting by twisted automorphisms. 

We say that a group $G$ {\em acts by twisted automorphisms} on a bialgebra $H$ if  a map of Cat-groups (as in section \ref{catgroups}): 
$(\tau,\theta):G\to\Aut_{\Tw}(H)$ is given: $$\tau:G\to Aut_\Tw(H),\quad \tau(g) = (g,F_g),$$ $$\theta:G\times G\to
Z(H)^\cross,$$ satisfying 
\begin{equation}\label{2coc}
f(\theta(g,h))\theta(f,gh) = \theta(f,g)\theta(fg,h),
\end{equation}
\begin{equation}\label{cobound}
f(F_g)F_f\Delta(\theta(f,g)) = (\theta(f,g)\otimes\theta(f,g))F_{fg}.
\end{equation}

A {\em crossed product} of a bialgebra $H$ with a group $G$ with respect to the action by twisted automorphisms
$(\tau,\theta)$ is a bialgebra $H*_{(\tau,\theta)}G$, which as a vector space is spanned by symbols $x*f$ for $x\in H$,
$f\in G$ subject to the linearity in the first argument $(x+y)*f = x*f + y*f$; with the product and the coproduct given
by the formulas: $$(x*f)(y*g) = (xf(y)\theta(f,g))*fg,$$ $$\Delta(x*f) = (\Delta(x)F_f^{-1})*(f\otimes f).$$ It is quite straightforward to check that all bialgebra axioms are satisfied. Indeed, associativity of multiplication follows from the condition (\ref{2coc}):
$$(x*f)((y*g)(z*h)) = (x*f)(yg(z)\theta(g,h)*gh) = xf(y)fg(z)f(\theta(g,h))\theta(f,gh)*fgh$$ 
coincides with 
$$((x*f)(y*g))(z*h) = (xf(y)\theta(f,g)*fg)(z*h) = xf(y)\theta(f,g)fg(z)\theta(fg,h)*fgh.$$ 
Coassociativity of comultiplication follows from the 2-cocycle property of twists: 
$$(\Delta\otimes I)\Delta(x*f) = (\Delta\otimes I)((\Delta(x)F_f^{-1})*(f\otimes f)) =$$ $$(\Delta\otimes I)\Delta(x)(\Delta\otimes I)(F_f)^{-1}(F_f\otimes 1)^{-1}*(f\otimes f\otimes f)$$ 
is equal to 
$$(I\otimes\Delta)\Delta(x*f) = (I\otimes\Delta)((\Delta(x)F_f^{-1})*(f\otimes f)) =$$ $$(I\otimes\Delta)\Delta(x)(I\otimes\Delta)(F_f)^{-1}(1\otimes F_f)^{-1}*(f\otimes f\otimes f).$$ 
Finally, the compatibility of multiplication and comultiplication follows from the condition (\ref{cobound}): 
$$\Delta((x*f)(y*g)) = \Delta(x*f)\Delta(y*g) = (\Delta(x)F_f^{-1}*(f\otimes f))(\Delta(y)F_g^{-1}*(g\otimes g)) =$$ 
$$\Delta(x)F_f^{-1}(f\otimes f)\Delta(y)(f\otimes f)(F_g)^{-1}(\theta(f,g)\otimes\theta(f,g))*(fg\otimes fg),$$ 
which, by the definition of twisted homomorphism, coincides with 
$$\Delta(x)\Delta(f(y))F_f^{-1}(f\otimes f)(f\otimes f)(F_g)^{-1}(\theta(f,g)\otimes\theta(f,g))*(fg\otimes fg).$$ 
At the same time 
$$\Delta(xf(y)\theta(f,g)*fg) = \Delta(x)\Delta(f(y))\Delta(\theta(f,g))F_{fg}^{-1}*(fg\otimes fg).$$

As an example we consider the case of a universal enveloping algebra. Let $A$ be a subspace of $(\Lambda^2\g)^\g$. Define a map of Cat-groups $A\to \Aut_\Tw(U(\g)[[h]])$ by assigning to $X\in A$ the twisted automorphism $(I,exp(Xh))$ and defining $\theta(X,Y)$ as $exp(\frac{1}{2}a(X,Y)h^2)$, where $a(X,Y)\in Z(U(\g))$ is a solution of $$[X,Y] = a(X,Y)\otimes 1 + 1\otimes a(X,Y) - \Delta(a(X,Y)),$$ satisfying the 2-cocyle condition $$a(X,Y) + a(X+Y,Z) = a(Y,Z) + a(X,Y+Z).$$ The crossed product $U(\g)[[h]]*A$ will have the following rules for multiplication and comultiplication: $$(x*X)(y*Y) = xy\ exp(\frac{1}{2}a(X,Y)h^2)*(X+Y),$$ $$\Delta(x*X) = \Delta(x)exp(Xh)*(X\otimes X).$$ Writing formally $X$ as $exp(l_Xh)$ the rules turn into the following $$exp(l_Xh)exp(l_Yh) = exp(a(X,Y)h^2 + (l_X + l_Y)h),$$ $$\Delta(exp(l_Xh)) = exp(Xh)(exp(l_Xh)\otimes exp(l_Yh)).$$ These can be resolved by setting 
\begin{equation}\label{commuta}
[l_X,l_Y] = -a(X,Y),
\end{equation}
\begin{equation}\label{comulti}
\Delta(l_X) = X + l_X\otimes 1 + 1\otimes l_X.
\end{equation}
We can formalise these equations by formally adding new generators $l_X,\ X\in A$  to $U(\g)$ subject to the relation (\ref{commuta}), with the comultiplication extending the one on $U(\g)$ and satifying (\ref{comulti}). The resulting object $U(\g)[A,a]$ is a bialgebra. Indeed, the only thing to check in this abstract setting is that the relation (\ref{commuta}) is preserved by the comultiplication: $$\Delta([l_X,l_Y] + a(X,Y)) = [X,Y] + [l_X,l_Y]\otimes 1 + 1\otimes [l_X,l_Y] + \Delta(a(X,Y)) =$$ $$([l_X,l_Y] + a(X,Y))\otimes 1 + 1\otimes ([l_X,l_Y] + a(X,Y)).$$

\begin{exam}
\end{exam}
Let $\g = {\cal H}(V,b) = V\oplus \langle c\rangle$ be the Heisenberg Lie algebra of the symplectic vector space $(V,b)$. Let $A$ be the space $V$ mapped into $(\Lambda^2\g)^\g$ by $v\mapsto v\wedge c$. As in the example \ref{heisen}, let $a(v,u) = \frac{b(u,v)}{3}c^3.$Then the extra generators $l_v,\ v\in V$ of the algebra $U({\cal H}(V,b))[V,a]$ satisfy $$[l_v,l_u] = \frac{b(u,v)}{3}c^3,$$ with the comultiplication defined by $$\Delta(l_v) = v\otimes c - c\otimes v + l_v\otimes 1 + 1\otimes l_v.$$

\end{document}